\documentclass[11pt,reqno]{amsart}

\usepackage{ulem}

\usepackage{hyperref}
\usepackage{fullpage}
\usepackage{graphicx}
\usepackage[final]{pdfpages}

\usepackage[all,cmtip]{xy} \usepackage{diagxy}
\newtheorem{theorem}{Theorem}[section]

\newtheorem{prop}[theorem]{Proposition}
\newtheorem{cor}[theorem]{Corollary}

\theoremstyle{definition}
\newtheorem{definition}[theorem]{Definition}

\theoremstyle{remark}

\numberwithin{equation}{section}

\newcommand{\defequal}{\stackrel{\text{\tiny def}}{=}}
\author{S. Aksoy}
\address{}
\author{A. Azzam}
\address{}
\author{C. Coppersmith}
\address{}
\author{J. Glass}
\address{}
\author{G. Karaali}
\address{}
\author{X. Zhao}
\address{}
\author{X. Zhu}
\address{}

\renewcommand{\emph}{\textit}

\thanks{Aksoy, Azzam, Coppersmith, and Karaali were partially supported by NSF Grant DMS-0755540. Karaali was partially supported by a Pomona College Hirsch Research Initiation Grant. Zhao was partially supported by the Hutchcroft Fund of the Department of Mathematics and Statistics at Mount Holyoke College. Zhu was partially supported by a Mount Holyoke College Ellen P. Reese Fellowship.}

\theoremstyle{definition}

\newcommand{\baseRing}[1]{\ensuremath{\mathbb{#1}}}

\newcommand{\N}{\baseRing{N}}

\newcommand{\s}{\mathfrak{s}}

\renewcommand{\phi}{\varphi}

\allowdisplaybreaks

\begin{document}

\title[Coalitions and Cliques]{Coalitions and Cliques in the School Choice Problem}

\date{July 25, 2011}

\begin{abstract}
The school choice mechanism design problem focuses on assignment mechanisms matching students to public schools in a given school district. The well-known Gale Shapley Student Optimal Stable Matching Mechanism (SOSM) is the most efficient stable mechanism proposed so far as a solution to this problem. However its inefficiency is well-documented; recently the Efficiency Adjusted Deferred Acceptance Mechanism (EADAM) was proposed as a remedy for this weakness. This note introduces two related adjustments to SOSM in order to address the same inefficiency issue. In one we create possibly artificial coalitions among students where some students modify their preference profiles in order to improve the outcome for some other students. Our second approach involves trading cliques among students where those involved improve their assignments by waiving some of their priorities. The coalition method yields the EADAM outcome as well as other Pareto dominations of the SOSM outcome, while the clique method yields all possible Pareto optimal Pareto dominations of SOSM. The clique method furthermore incorporates a natural solution to the problem of breaking possible ties within preference and priority profiles. We discuss the practical implications and limitations of our approach in the final section of the article.
\end{abstract}

\maketitle

\bibliographystyle{plain}

\section{Introduction}
\label{S:Introduction}

Since the mid-eighties, in cities across the United States, public school assignment policies have shifted towards providing students the opportunity to influence their school assignment. The main objective of these {\it school choice} policies is to allow all students to attend more desirable schools. A standard theoretical framework for studying such policies is two-sided matching (cf.\ \cite{Ga01, RoSo90}).
Presented in this context, the goal of the School Choice Problem (SCP) is to devise a matching mechanism (designed by or for the school district) that allocates available resources (seats in schools) among players (students or parents) subject to district priorities and legal requirements. 

The ideal way to solve the SCP would be to make all schools desirable to sufficiently many students so that all students could attend schools high on their preference profile. Short of a magic wand, we follow in the footsteps of other researchers in our attempt to create the most desirable matching possible given the seemingly intractable problem of too few seats in desirable schools. Current school choice mechanisms tolerate a large number of students receiving low preference schools (``inefficiency'') in order to respect school priority structures (``stability''). The ultimate purpose of these priorities is to benefit the students, but in many practical situations they are also the direct cause of the efficiency losses. This suggests that taking a stable solution as baseline (starting out with a balanced focus on school priorities and student preferences) and then making improvements for efficiency (emphasizing preferences over priorities at the expense of stability) may be a good compromise  incorporating both preferences and priorities,  resulting in more desirable matchings.\footnote{A relevant quote from \cite{AbPaRo09}: ``Pareto efficiency for the students is the primary welfare goal, but [...] stability of the matching, and strategy-proofness in the elicitation of student preferences, are incentive constraints that likely have to be met for the system to produce substantial welfare gains over the [current] system.''}

To this end we employ the language and methods of mechanism design as applied to the SCP \textit{a la} \cite{AbSo03}. In our context  the designer/principal is the school district (or whoever is choosing the mechanism to be used). Students are the players; schools are merely items to be consumed, though the end result
is influenced by the priorities defined by the district. The designer's desired outcome is that all students get matched to a school high on their preference lists. Thus the underlying sentiment in our research coincides with that found in previous research in the SCP: the idea that the process should result in as many students as possible receiving placements at schools as high on their preference list as possible in a given matching. School districts may of course have additional motives when designing their policies, which might include, among other things, diversity and social justice considerations; these can to an extent be incorporated in the school priority structures.

In this article we examine several commonly applied mechanisms in two sided matching and school choice.  We introduce two related approaches and examine how these new approaches measure up.  In Section \S\S\ref{SS:Background} we introduce three standard mechanisms used in this area of investigation: SOSM\footnote{Since SOSM basically uses a particular type of deferred acceptance (DA) procedure applied to the SCP, in \S\S\ref{SS:Background}, we also describe the DA algorithm briefly.}, EADAM, and TTC.  In \S\ref{S:Coalitions} we introduce our first new approach by studying the impact on outcomes if students were to form ``coalitions'' in order to affect their school assignments.  This section closely follows  \cite{Hu06} where it is shown that while the Gale-Shapley deferred acceptance algorithm (DA) disincentivizes strategic action by individuals, it is still feasible for groups to beat the system by coming together and strategizing. We adapt Huang's methods to the SCP and along the way prove that SOSM (DA as applied to School Choice) is not coalition-strategyproof. We then show that coalitions created by the school district (or any designer/principal) could result in efficiency gains over the DA/SOSM outcome. We find that this approach produces many possible Pareto improvements to the DA/SOSM outcome and, in particular, will yield the Efficiency Adjusted Deferred Acceptance Mechanism (EADAM) outcome as one of its outcomes.\footnote{All acronyms in this paragraph will be explained in detail in \S\S\ref{SS:Background}.}

Following up on the coalition/cooperation theme, in \S\ref{S:Cliques} we focus on groups of students who form trading cycles (``cliques'') to improve their own assignments.\footnote{The term \emph{clique} has a specific meaning in graph theory, unrelated to our work here.}  We examine the impact of these trading cliques when starting with the baseline assignment that results from the DA/SOSM mechanism.  In particular we show that the coalition improvements of \S\ref{S:Coalitions} can be integrated into this new framework, which proves to be a powerful construct to study cycle improvements of various kinds. We also note that indifferences in student preferences may be incorporated into this model. Although a considerable amount of research has been done regarding indifferences within school priority classes, indifference in student preferences has not been studied in as much depth. As far as we know, this characteristic of cycle improvement models has not been investigated before. 

Interwoven throughout this work is our emphasis on student preferences as opposed to school priorities, that is, efficiency as opposed to stability. This is due to our wariness of accepting the cost of upholding priorities at expense to the students they are purported to help. We also note that to run the standard algorithms, strict priority rankings are needed.
While priorities vary between districts,  a single priority class will often have a large number of students. Thus to apply standard two-sided matching algorithms to the SCP, one must ultimately break ties randomly among  students within a single priority class. This creates arbitrary rankings, introduces artificial conditions, and results in a sizable efficiency loss (see \cite{ErEr08} for a study of tie-breaking and its efficiency cost). Both collaborative approaches presented (coalitions and cliques) make efficiency adjustments to a stable baseline solution which we see as a feasible way to partially address this tie-breaking conundrum as well as the more standard stability-efficiency tradeoff mentioned earlier.

%%%%%%%%%%%%%%%%%%%
%%%%%%%%%%%%%%%%%%%
%%%%%%%%%%%%%%%%%%%

\subsection{Notation and basic terms used}
\label{SS:Notation}

Let $I$ denote a nonempty set of students, and $S$ a nonempty set of schools. A {\bf matching} $M: I\to I\times S$ is a function that associates every student with exactly one school, or potentially no school at all. Write $\mathfrak{M}$ for the set of matchings. We will also use the related function $M : I \rightarrow S$ and write $M[i] = s$ if $M(i) = (i,s)$.

A {\bf preference profile} ${P}_i$ for student $i \in I$ is a tuple $(S_{1},\dots,S_{n})$ where the $S_{j}$'s form a partition of $S$ and every element of $S_{j}$ is preferred to every element of $S_{k}$ if and only if $j<k$.
Define the {\bf ranking function} $\phi_i : S \rightarrow \N$  of a student $i \in I$ by letting
$\phi_i(s)$ denote $i$'s {ranking} of $s \in S$. In other words $\phi_i(s) = j$ if $s \in S_j$.
When each $S_{j}$ is a singleton, we say that $i$'s preference profile is {\bf strict}, (in which case we can view $P_i$ as an $n$-vector).  If $s_{k},s_{l}\in S_{j}$ for some $j,k\neq l$, then we say that the student is {\bf indifferent} between $s_{k}$ and $s_{l}$. If $i$ prefers $s_{k}$ to $s_{l}$, we write $s_k \succ_i s_l$, or simply $s_k \succ s_l$ if $i$ is unambiguous. Note that the notation $\succ$ denotes a strict order; if we want to describe a weak order, we will write $\succeq$. We denote a set consisting of preference profiles for each student in $I$ by ${\bf P} = \{P_i : i \in I\}$ and the space of all such sets is denoted by $\mathfrak{P}$.

A {\bf priority structure} ${\Pi}_s$ for school $s \in S$ is a tuple $(I_{1},\dots,I_{n})$ where the $I_{j}$'s form a partition of $I$ and every element of $I_{j}$ is preferred to every element of $I_{k}$ if and only if $j < k$. When each $I_{j}$ is a singleton, we say that $s$'s priority structure is {\bf strict}, (in which case we can view $\Pi_s$ as an $n$-vector).  If $i_{k},i_{l}\in I_{j}$ for some $j,k\neq l$, then we say that the school is {\bf indifferent} between $i_{k}$ and $i_{l}$. If $s$ prefers $i_{k}$ to $i_{l}$ we write
$i_k \succ_s i_l$, or simply $i_k \succ i_l$ if $s$ is unambiguous. Once again, the notation $\succ$ denotes a strict order; if we want to describe a weak order, we will write $\succeq$. We denote a set consisting of priority structures for each school in $S$ by ${\bf \Pi} = \{\Pi_s : s \in S\}$ and the space of all such complete sets is denoted by $\mathfrak{\Pi}$.

A matching $M'$ {\bf (Pareto) dominates} $M$ if $M'[i] \succeq_i M[i]$ for all $i$ and $M'[j] \succ_j M[j]$ is strict for some $j$. A {\bf (Pareto) efficient} {\bf matching} is a matching that is not (Pareto) dominated.

A {\bf matching mechanism} $\mathcal{M}: \mathfrak{P} \times \mathfrak{\Pi} \to \mathfrak{M}$
is a function that takes an ordered pair $(\textbf{P}, {\bf \Pi}) \in \mathfrak{P} \times \mathfrak{\Pi}$ of preferences and priorities and produces a matching.

Let $\Pi_{s}$ be a priority structure for school $s$. We say that a matching $M$ {\bf violates the priority} of $i\in I$ for $s$ if there exist some $j\in I$ and $s^{\prime} \in S$ such that
\begin{enumerate}
\item $M[j]=s$, $M[i] = s^{\prime}$: $j$ gets assigned $s$ under $M$ and $i$ gets assigned $s^{\prime}$ under $M$.
\item $s \succ_i s^{\prime}$: $i$ prefers attending $s$ over $s^{\prime}$  and
\item  $i \succ_s j$: $s$ prioritizes $i$ over $j$.
\end{enumerate}
We say that a matching $M$ is {\bf stable} if
\begin{enumerate}
\item $M$ does not violate any priorities.
\item No student is matched to a lower-ranked school  when a more preferred school is unfilled.
\end{enumerate}
A {\bf stable mechanism} is one that always produces stable matchings.
A matching mechanism is \textbf{strategyproof} if there is no rational  incentive for a student to misrepresent their preferences.

%%%%%%%%%%%%%%%%%%%
%%%%%%%%%%%%%%%%%%%
%%%%%%%%%%%%%%%%%%%

\subsection{Background}
\label{SS:Background}

In this section we introduce several well-known matching algorithms / mechanisms and give a few illustrative examples.\footnote{We should remark that all the mechanisms described in this section use strict preference lists for students.}  
In the literature on two-sided matching the Gale-Shapley deferred acceptance algorithm  \cite{GaSh62} is highly touted, see \cite{RoSo90} for an extensive review of the various applications of this algorithm and \cite{Ro08} for a more recent historical overview. Gale and Shapley first described their deferred acceptance method in the context of the \emph{stable-marriage problem}: There are two distinct groups (men and women) each with an individual preference profile ranking the members of the opposite sex; the ultimate goal is to find a stable matching between the men and the women.\footnote{In this context stability means that no man will prefer a woman other than his own partner who also prefers him more than the man to whom she was matched.} The deferred acceptance procedure (DA) is as follows:  

In round $1$ each man proposes to his top choice. Each woman then tentatively accepts the man who is highest on her preference list among those who proposed to her that round (who is now her fiance) and rejects the rest. In step $k$, each unengaged man proposes to his next choice, and each woman considers her new suitors along with her current fiance, tentatively accepting her top choice among them, and rejecting the rest.  The algorithm ends when all men are engaged.

In \cite{GaSh62}, Gale and Shapley proposed applying their deferred acceptance algorithm to the college admissions problem, with the men replaced by  students applying to colleges and the women replaced by the colleges to which they applied. In 2003 Abdulkadiro\v{g}lu and S\"{o}mnez adapted the Gale-Shapley algorithm to the SCP \cite{AbSo03} calling it the Student Optimal Stable Mechanism (SOSM). The Gale-Shapley deferred acceptance algorithm, as adapted to the SCP in the form of SOSM, is widely held to be a practical mechanism for implementation. In particular, several large districts such as New York City and Boston \cite{APR05, AbPaRo09, APRS05, AbPaRoSo06} have adopted SOSM as their mechanism of choice. Pareto efficiency, stability, and strategyproofness are the main criteria used to evaluate a school choice matching mechanism,\footnote{We will add a fourth criterion to our consideration, see \S\S\ref{SS:NewCriterion}}  and within these measures, SOSM performs well. Indeed, SOSM offers a stable strategyproof mechanism whose outcomes Pareto dominate all other stable matchings.\footnote{Note that a stable mechanism can never really be strategy-proof in the complete sense. More specifically, \textit{no stable matching mechanism exists for which stating the true preferences is always a best response for every agent where all other agents state their true preferences} (see for instance \cite[Cor.\ 4.5]{RoSo90}). However the DA/SOSM is practically strategy-proof as we only view the students as strategic players and the student optimality implies that there is no incentive for the students to misrepresent their preferences (cf.\ \cite{Ro82}). This perspective does not take into account manipulation by schools in capacity (cf.\ \cite{So97}) or preferences, see \cite{Eh10} for recent work addressing these issues.}

In this article we use DA/SOSM as a baseline to improve upon.  Indeed certain improvements are possible, feasible, and desirable because SOSM suffers from documented efficiency and welfare losses \cite{AbPaRo09, Ke10}. More precisely, although the SOSM outcome in a given setting dominates any other stable matching, it can often be dominated by an unstable matching. We now illustrate this potential trade-off between stability and efficiency with an example due to Roth, which we will label as $SCP_1$. Assume there are three schools, $s_{1}, s_{2}, s_{3}$ and three students $i_{1}, i_{2}, i_{3}$. The priorities of the schools and the preferences of the students are given by:
\[ SCP_1 : \qquad \begin{array}{ccc}
i_{1}:s_{2} \succ s_{1} \succ s_{3} & \qquad & s_{1}: i_{1} \succ i_{3} \succ i_{2} \\
i_{2}: s_{1}  \succ  s_{2} \succ s_{3} & \qquad & s_{2}:i_{2} \succ i_{1} \succ i_{3} \\
i_{3}: s_{1} \succ s_{2} \succ s_{3} & \qquad &  s_{3}:i_{2} \succ i_{1} \succ i_{3}
\end{array}\]
where $a \succ b$  stands for ``{\it $a$ is preferable to $b$}" (more about notation and terminology can be found in \S\S\ref{SS:Notation}).
Here, the only stable matching is:
\[
M_S^{SCP_1}  = \begin{pmatrix}i_{1}& i_{2}& i_{3}\\ s_{1}&s_{2}&s_{3}\end{pmatrix},\]
but this matching is (Pareto) dominated by:
\[
M_E^{SCP_1}  = \begin{pmatrix}i_{1}& i_{2}& i_{3}\\ s_{2}&s_{1}&s_{3}\end{pmatrix}.\]
We see that $M_E^{SCP_1}$ (Pareto) dominates $M_S^{SCP_1}$ because it assigns $i_1$ and $i_2$ schools they prefer over their $M_S^{SCP_1}$ assignment. Furthermore $M_E^{SCP_1}$ is (Pareto) efficient. However the matching is no longer stable because $i_2$ is in the position of violating $i_3$'s priority for $s_1$.

In order, in part, to address the weakness illustrated by the example above, Kesten in \cite{Ke10} proposes a new mechanism, and calls it the Efficiency Adjusted Deferred Acceptance Mechanism (EADAM). In order to understand EADAM we must first define an \textbf{interrupter}. Let student $i$ be one who is tentatively placed in a school $s$ at some step $t$ while running the SOSM, and rejected from it at some later step $t^{\prime}$. If there exists at least one other student who is rejected from school $s$ after step $t-1$ and before step $t^{\prime}$, then we call student $i$ an \textbf{interrupter} for school $s$ and the pair $(i, s)$ is an \textbf{interrupting pair} of step $t^{\prime}$. An interrupter is \textbf{consenting} if she allows the mechanism to violate her priorities at no expense to her. The EADAM then runs as follows:
\begin{itemize}
\item \underline{Round 0}: Run the SOSM.
\item \underline{Round 1}: Find the last step (of the SOSM run in Round 0) at which a consenting interrupter is rejected from the school for which he/she is an interrupter. Identify all interrupting pairs in that step which contain a consenting interrupter. If there are no such pairs, then stop. Otherwise for each identified interrupting pair $(i,s)$, remove school $s$ from the preference list of student $i$ without changing the relative order of the remaining schools. Rerun the SOSM with the new preference profile for $i$ until all students have been assigned.
\item[] And in general,
\item \underline{Round k}, k $\ge$ 1: Find the last step (of the SOSM run in the previous round) at which a consenting interrupter is rejected from the school for which he/she is an interrupter. Identify all interrupting pairs in that step which contain a consenting interrupter. If there is no such pair, stop. Otherwise for each identified interrupting pair $(i, s)$, remove school $s$ from the preference list of student $i$ without changing the relative order of the remaining schools. Rerun the SOSM with the new preference profile until all students have been assigned.
\end{itemize}

In $SCP_1$, $(i_3, s_1)$ is an interrupting pair and the EADAM with the consent of $i_3$ outputs the Pareto efficient matching $M_E^{SCP_1}$.

Even though the end result of consenting for interrupters is that they allow the mechanism to violate their priorities for schools they would not be assigned anyway, the step-by-step description above points toward a different route of obtaining the same outcome. The practical outcome would be the same if the consenting interrupters were to modify their preference lists in such a way as to drop schools that they'd not have been assigned to anyway. Thus instead of asking students to sign consent forms to waive priorities, as would be required to run the EADAM, we could in theory ask them to reconsider their preference lists.\footnote{In reality this is not desirable; we emphatically want students to be truthful in declaring their preferences.}

Last, we describe the Top Trading Cycles (TTC) mechanism, first introduced in \cite{ShSc74} (also see \cite{AbSo99, KrWa07}) and adapted to the school choice context in \cite{AbSo03} as an alternative to SOSM. TTC is a strategyproof mechanism that compromises on stability to achieve efficiency, 
and proceeds as follows:
\begin{itemize}
\item \underline{Round 1}: Each student points to his or her first choice school. Similarly each school points to its first choice applicant. Since there are finitely many students and schools, there is at least one cycle. For each such cycle do the following: Assign each student in the cycle to the school he or she is pointing to and remove the student and the school from the market. All unassigned students and unfilled schools move on to the next round.
\item[] And in general,
\item \underline{Round k}, k $\ge$ 1: Each unassigned student points to his or her top choice school among the unfilled ones. Each unfilled school points to the student whom it ranks highest among the unassigned students. There should be at least one cycle. For each such cycle do the following: Assign each student in the cycle to the school he or she is pointing to and remove the student and the school from the market. All unassigned students and unfilled schools move on to the next round. The algorithm runs until all students have been assigned.
\end{itemize}
Thus in essence, once in a trading cycle, students are allowed to trade schools among themselves.

%%%%%%%%%%%%%%%%%%%
%%%%%%%%%%%%%%%%%%%
%%%%%%%%%%%%%%%%%%%

\subsection{A new evaluation criterion}
\label{SS:NewCriterion}
The three main criteria most commonly used to evaluate school choice mechanisms are stability, strategyproofness and (Pareto) efficiency. In \cite{AACGKZZha10} we introduced a new student-optimal criterion, a ``\textit{preference reverence index}", and showed that it incorporates a measure of student optimality that is not fully captured by these three previously emphasized criteria.  When evaluating matching outcomes in the later parts of this article we make use of this index, so we provide a brief exposition about it in this section.

With the notation from \S\S\ref{SS:Notation} we define $\mu: \mathfrak{M} \to \N$ by
\[
\mu(M)=\sum_{i \in I}|\phi_i(M[i])-1|.\]
For any given $M \in \mathfrak{M}$ we will call $\mu(M)$ the {\bf preference reverence index} of $M$ or simply the {\bf preference index}.
We summarize some results regarding the preference index in the following:
\begin{prop}The following are some properties and implications of the preference reverence index as applied to the SCP:
\begin{enumerate}
\item There can exist two stable matchings with the same preference index.
\item The stable matching with the smallest preference index is the SOSM outcome and it is the unique stable matching with that preference index.
\item If a matching $M$ (Pareto) dominates $M^{\prime}$, then $M$ has a lower preference index.
\end{enumerate}
\end{prop}

See \cite{AACGKZZha10} for more on the preference reverence index.

%%%%%%%%%%%%%%%%%%%
%%%%%%%%%%%%%%%%%%%
%%%%%%%%%%%%%%%%%%%

\section{Coalitions in the school choice problem}
\label{S:Coalitions}

In \cite{Hu06} Huang discusses a weakness found in the Gale-Shapley stable matching in the context of the stable marriage problem and introduces the idea of {\it coalition cheating in the marriage problem}. More specifically he shows that a coalition can be formed where some men, without forgoing their own Gale-Shapley stable matching assignment, can cheat (misrepresent their preferences) so that some other men marry women who are higher on their preference list.

In this section we apply these ideas to the school choice problem. In \S\S\ref{SS:HuangsConstruction} we give the details of Huang's construction.  Then in \S\S\ref{SS:CoalitionImplementation} we introduce the elements of cheating coalitions in the context of the SCP, and  discuss some implementation issues. In \S\S\ref{SS:Hopeless} we focus on some interesting theoretical consequences of coalitions in the context of the SCP.  In \S\S\ref{SS:CoalitionsEADAM} and \S\S\ref{SS:CoalitionsTTC},
we compare the possible outcomes of coalitions to that of EADAM, and to that of TTC, respectively.

%%%%%%%%%%%%%%%%%%%
%%%%%%%%%%%%%%%%%%%
%%%%%%%%%%%%%%%%%%%

\subsection{Huang's Construction and Coalitions}
\label{SS:HuangsConstruction}

Originally proved in \cite{DuFr81}, the following theorem establishes that in the stable marriage problem, there exists no coalition of men that may falsify their preferences such that every member of the coalition receives a \emph{strictly better} assignment:
\begin{theorem}[Dubins-Freedman 1981]\label{T:DubinsFreedman}
In the Gale-Shapley men-optimal algorithm, no subset of men can improve their assignment by falsifying their preference lists.
\end{theorem}
Therefore in order to study coalitions which falsify preferences to improve their assignments, Huang introduces a nuanced notion of coalitions, which incorporates a separation between two main components: those who falsify their preferences, and those that benefit from these falsifications. In the following we provide a detailed exposition of his construction.

Let $\mathcal{M}en$ and $\mathcal{W}om$ be the set of men and women respectively in a given stable marriage problem. Let $M$ be the Gale-Shapley stable matching assignment for this problem when all members of $\mathcal{M}en$ and $\mathcal{W}om$ submit their true preferences.
A {\bf coalition} $C$ is defined in terms of a pair $(K,A)$ of subsets of the set $\mathcal{M}en$. The first subset, the {\bf cabal} $K = (m_{1},m_{2}, ...,m_{|K|})$ of the coalition $C$, is a list of men such that each man $m_{i}$, $1 \leq i \leq |K|$, prefers $M[m_{i-1}]$ to his own partner $M[m_{i}]$, indices taken module $|K|$. In other words, we have $M[m_{i-1}] \succ_{m_i} M[m_i]$ for $1 \leq i \leq |K|$, what we will call a {\bf cabal loop}, written $(m_1 \rightarrow m_{|K|} \rightarrow m_{|K|-1} \rightarrow \cdots \rightarrow m_1)$, a closed chain of men each of whom would prefer the stable partner of the person before him to his own partner.
The second subset, the {\bf accomplice set} $A=A(K)$ of cabal $K$, is a set of men $A(K) \subset \mathcal{M}en$ such that $m \in A(K)$ if for some $m_{i} \in K$, $M[m_{i}] \succ_{m} M[m]$ and $m \succ_{M[m_{i}]} m_{i+1}$. In other words, an {\bf accomplice} is a man who in his truthful preference list ranks the stable partner of someone in the cabal higher than his own stable partner, while he himself is ranked higher by that woman than the next member of the cabal who would prefer her to his own partner. Note that $K$ and $A(K)$ may or may not be disjoint.

For any given man $m \in \mathcal{M}en$ we can write the preference profile of $m$ as a disjoint union of three sets: $(P_L[m],M[m],P_R[m])$. Here the set $P_L[m]$ (respectively $P_R[m]$) is simply the list of women on $m$'s preference profile to the left (respectively to the right) of his stable partner $M[m]$.

Let $C = (K, A(K))$ be a coalition as described above and let $\pi_r$ denote a random permutation of $\mathcal{W}om$. The coalition cheating procedure proceeds as follows (\cite[Thm.2]{Hu06}):
Each accomplice $m \in A(K)$ submits a falsified list of the form
\[ (\pi_r(P_L[m]-X), M[m], \pi_r(P_R[m]\cup X)).\]
Here
$X$ is the set
\[ \{w \in M[K] \vert w = M[m_i], w \succ_m M[m], m \succ_w m_{i+1} \}\]
if $m \not \in K$, {and}
\[ X = \{w  \in M[K] \vert w = M[m_i], i \neq j, w \succ_{m_j} M[m_{j}], m_j \succ_w m_{i+1}\},\]
if $m = m_j \in A(K) \cap K$. In other words, accomplices modify their preference profiles by moving women on the left of their stable partner to the right of their stable partner if they are desirable to other men in the cabal. In particular if $m$ is an accomplice, then the set $X$ of women $m$ moves to the right of his stable partner will consist of all the stable partners of members of the cabal who rank $m$ higher than the man following their stable partners in the cabal loop.
Huang then shows that in the resulting matching $M^{\prime}$, $M^{\prime}[m_i] = M[m_{i-1}]$ for $m_i \in K$ and $M^{\prime}[m] = M[m]$ for $m \not \in K$.
Note that in the above the falsified preference lists incorporate a random permutation $\pi_r$ of the preferences to the left and the right of the stable partner. The cheating coalition procedure is quite robust, in that such a random permutation will not affect the outcome. In other words, the resulting matching creates a cyclical reassignment of those within the cabal loop while leaving all other assignments as they were.

%%%%%%%%%%%%%%%%%%%
%%%%%%%%%%%%%%%%%%%
%%%%%%%%%%%%%%%%%%%

\subsection{Coalitions and school choice}
\label{SS:CoalitionImplementation}

Here is an analogue of 
Theorem \ref{T:DubinsFreedman} in the SCP context:
\begin{theorem}\label{T:DubinsFreedmanSCP}
In the SOSM algorithm, no subset of students can improve their assignment by falsifying their preference lists.
\end{theorem}
This is particularly easy to see in our previous example.
Assume there are three schools, $s_{1}, s_{2}, s_{3}$ and three students $i_{1}, i_{2}, i_{3}$. The school priorities and the student preferences are given by:

\[SCP_1:\qquad  \begin{array}{ccc} 
i_{1}:s_{2} \succ s_{1} \succ s_{3} & \qquad & s_{1}: i_{1} \succ i_{3} \succ i_{2} \\ 
i_{2}: s_{1}  \succ  s_{2} \succ s_{3} & \qquad & s_{2}:i_{2} \succ i_{1} \succ i_{3} \\
i_{3}: s_{1} \succ s_{2} \succ s_{3} & \qquad &  s_{3}:i_{2} \succ i_{1} \succ i_{3}
\end{array}\]
Recall that the SOSM yields the (unique stable) matching \[
M_S^{SCP_1}  = \begin{pmatrix}i_{1}& i_{2}& i_{3}\\ s_{1}&s_{2}&s_{3}\end{pmatrix},\]
Each non-singleton, nonempty subset of the set of students corresponds to a valid coalition:
\[c_{1}=\{i_{1},i_{2}\}\ \ c_{2}=\{i_{1},i_{3}\}\ \  c_{3}=\{i_{2},i_{3}\}\ \ c_{4}=\{i_{1},i_{2}, i_{3}\}\] 
For coalition $c_{1}$, the only way for both students to improve their lot is for $i_{1}$ to receive $s_{2}$ and $i_{2}$ to receive $s_{1}$. However, unless $i_{3}$ receives a better assignment with this coalition,
the resultant matching will be unstable since $i_{2}$ violates $i_{3}$'s priority at $s_{2}$. So $i_{3}$ must receive a better school if $i_{1}$ or $i_{2}$ have a chance of receiving a better assignment, so $i_{3}$ must receive $s_{1}$ or $s_{2}$. 
Thus clearly $i_{1}$ and $i_{2}$ cannot both strictly benefit. The arguments for the remaining coalitions are similar. 

The underlying intuition for the above result is much clearer than brute force computation. For a coalition of students to benefit from their falsification under SOSM, either their interests conflict, or the constraint of stability requires appeasing other students not in the coalition, and this conflicts with the interests of those in the coalition. For a deeper and more enlightening discussion, see \cite{DuFr81}. With this example and motivating theorem in mind, we now proceed to formalize our coalition model for the SCP, which we will call the Coalitional Improvement Mechanism (CIM). 

Let $I$ and $S$ be the set of students and schools respectively in a given SCP. Let $M$ be the SOSM stable matching assignment for the case where all students submit their true preferences.
A {\bf coalition} $C$ is defined in terms of a pair $(K,A)$ of subsets of the set $I$ of students. The first subset, the {\bf cabal} $K = (i_{1},i_{2}, ...,i_{|K|})$ of a coalition $C$, is a list of students such that each student $i_{k}$, $1 \leq k \leq |K|$, prefers $M[i_{k-1}]$ to $M[i_{k}]$, indices taken modulo $|K|$.  In other words, we have $M[i_{k-1}] \succ_{i_k} M[i_k]$ for $1 \leq k \leq |K|$, and a {\bf cabal loop}, written $(i_1 \rightarrow i_{|K|} \rightarrow i_{|K|-1} \rightarrow \cdots \rightarrow i_1)$, a closed chain of students each of whom would prefer the stable assignment of the person before him to his own stable assignment.
The second subset, the {\bf accomplice set}
$A=A(K)$ of cabal $K = (i_{1},i_{2},...,i_{|K|})$, is a set of students $A(K) \subset I$ such that $i \in A(K)$ if for some $i_{k} \in K$, $M[i_{k}] \succ_{i} M[i]$ and $i \succ_{M[i_{k}]} i_{k+1}$. In other words, an {\bf accomplice} is a student who in his truthful preference list ranks the stable assignment of someone in the cabal higher than his own stable assignment, while he himself is ranked higher by that school than the next member of the cabal who would prefer it to his own school. Note that $K$ and $A(K)$ may or may not be disjoint.

For any student $i \in I$ we can write the preference profile of $i$ as a disjoint union of three sets: $(P_L[i],M[i],P_R[i])$. Here the set $P_L[i]$ (respectively $P_R[i]$) is simply the list of schools on $i$'s preference profile to the left (respectively to the right) of his stable assignment $M[i]$. Let $\pi_r$ denote a random permutation of $S$. The coalition cheating procedure (CIM) is described in: 
\begin{theorem}[cf.\ Huang 2006 \cite{Hu06}]
\label{T:SCPCoalition}
Let $M$ be the SOSM matching for a given SCP when students submit their true preferences. Consider a coalition $C = (K, A(K))$, and suppose that each accomplice $i \in A(K)$ submits a falsified list of the form $(\pi_r(P_L[i]-X), M[i], \pi_r(P_R[i]\cup X))$,
where
\begin{itemize}
\item if $i \not \in K$, then $X = \{s \in M[K] \vert s = M[i_k], s \succ_i M[i], i \succ_s i_{k+1} \}$, {\sc and}
\item if $i = i_k \in A(K) \cap K$, $X = \{s \in M[K] \vert s = M[i_j], j \neq k, s \succ_{i_k} M[i_{k}], i_k \succ_s i_{j+1}\}$.
\end{itemize}
Then in the resulting matching $M^{\prime}$, $M^{\prime}[i_k] = M[i_{k-1}]$ for $i_k \in K$ and $M^{\prime}[i] = M[i]$ for $i \not \in K$.
\end{theorem}
The proof of Theorem \ref{T:SCPCoalition} is an easy adaptation from the analogous result of Huang \cite{Hu06}.

Accomplices modify their preference profiles by moving schools on the left of their stable assignment to the right of their stable assignment if they are desirable to other students in the cabal. In particular if $i$ is an accomplice, then the set $X$ of schools $i$ moves to the right of his stable assignment will consist of all the stable assignments of the members of the cabal that rank $i$ higher than the student following their stable assignment in the cabal loop. Again, this procedure is robust, i.e., a random permutation of the two sides of the stable assignment will not affect the outcome.

Let us now consider an example which we will label $SCP_2$. Let $I = \{i_1, i_2, i_3, i_4, i_5\}$ and $S = \{s_1,s_2,s_3,s_4,s_5\}$ be the set of students and schools, respectively, and let their respective preference and priority profiles be given as follows:
\[ SCP_2 : \qquad \begin{array}{ccc}
i_1: s_2 \succ s_5 \succ s_4 \succ s_3 \succ s_1 & \qquad  & s_1:i_3 \succ i_2 \succ i_4 \succ i_1 \succ i_5\\
i_2: s_2 \succ s_5 \succ s_4 \succ s_1 \succ s_3 & \qquad  &  s_2: i_4  \succ  i_5\succ i_1 \succ i_2 \succ i_3\\
i_3: s_5 \succ s_2 \succ s_1 \succ s_3 \succ s_4 & \qquad  & s_3: i_2 \succ i_3 \succ i_4 \succ i_5 \succ i_1\\
i_4: s_4 \succ s_1 \succ s_2 \succ s_3 \succ s_5 & \qquad  & s_4: i_1 \succ i_2 \succ i_3 \succ i_5 \succ i_4\\
i_5: s_5 \succ s_4 \succ s_2 \succ s_3 \succ s_1 & \qquad  & s_5: i_1 \succ i_2 \succ i_5 \succ i_3 \succ i_4
\end{array}\]

Note that the matching output by the SOSM for $SCP_2$ is:
\[ M_S^{SCP_2} = \begin{pmatrix}
i_1&i_2&i_3&i_4&i_5 \\
s_5&s_4&s_1&s_2&s_3
\end{pmatrix}\]
and has preference reverence index 10 (\S\S\ref{SS:NewCriterion}).

We now consider the following coalition $C = (K, A(K))$: Let $K = \{i_1,i_2,i_4\}$ with the cabal loop $(i_1 \rightarrow i_4 \rightarrow i_2 \rightarrow i_1)$. The accomplice set $A(K)$ is $\{i_5\}$ and the set $X$ for $i_5$ is $\{s_2,s_4\}$. In other words the only student who modifies his preference profile is $i_5$. We display his old and new profiles:
\[ i_5 \text{'s old profile}: s_5 \succ s_4 \succ s_2 \succ \underline{s_3} \succ s_1\]
\[ i_5 \text{'s new profile}: s_5 \succ \underline{s_3} \succ s_1 \succ{s_2} \succ s_4 \]
[We underlined $i_5$'s stable assignment $s_3$.] The outcome matching when we rerun SOSM is:
\[ M_C^{SCP_2} = \begin{pmatrix}
i_1&i_2&i_3&i_4&i_5 \\
s_2&s_5&s_1&s_4&s_3
\end{pmatrix}\]
which improves the outcome for all members of the cabal, does not affect the remaining students, and has preference index 6.\footnote{We note that this is also the EADAM outcome if $i_5$ consents.} We will discuss this example further in \S\S\ref{SS:CliquesDescription}.

%%%%%%%%%%%%%%%%%%%
%%%%%%%%%%%%%%%%%%%
%%%%%%%%%%%%%%%%%%%

\subsection{Selfless and hopeless students}
\label{SS:Hopeless}

In the  example above,
student $i_5$ modified his preference profile in order to change the group outcome to $M_C^{SCP_2}$ from $M_S^{SCP_2}$, but in the end, he did not improve his own assignment. In fact, in any coalition the accomplice set will include some students who do not benefit from the coalition. To understand why, we go back once again to the stable marriage problem.  In that context, Dubins and Freedman  \cite{DuFr81} showed that it is impossible for every man in a coalition to improve his assignment. In other words a subset of men cannot falsify their preference lists so that all of them get  better partners. 
Thus in Huang's coalition cheating framework, in order to improve the assignments of some members of the cabal, there have to be some selfless men who are willing to make adjustments to their preference profiles despite the fact that they themselves cannot benefit from the arrangement. Analogously in the SCP context, for cheating coalitions to work, there have to be selfless students who are willing to adjust their preferences even though this will not improve their assignment.  It should be noted that in these cases, the selfless participant does not end up with a less desirable assignment.

This situation raises the question of the feasibility of coalition cheating, as we now see that some members of the accomplice set have no obvious incentive to make the required adjustments to their preference profiles. As a possible resolution to this issue, Huang \cite{Hu06} proposes a randomized strategy in which every man can expect to marry women ranked higher on their preference lists. This strategy requires men to take risks, since some men can end up with less desirable partners. In the School Choice Problem however, students and parents are relatively risk-averse. Thus a coalition agreeing to a randomized arrangement like the one mentioned by Huang in the Stable Marriage Problem is not rationally practical in the SCP.
Moreover, in school districts where there may be many students competing against each other for a limited number of seats at desirable schools, cooperation among parents and students is not plausible. 

Nonetheless we propose that our work does not merely present a theoretical framework to investigate collaboration and cooperation issues in the context of the SCP, but in fact it can have implementable outcomes in this context. More specifically, we propose that a school district, given perfect information of all student preferences, can create ``artificial coalitions" that would result in a more efficient outcome than that of SOSM. Using the SOSM matching as a baseline, a computer program could identify all matchings that would result from all possible coalitions.\footnote{However, arguably, the computer power needed for this could be quite large.}  It would then be up to the district to decide which coalitions would be most appropriate to further its own district goals. For instance one district could select the coalition(s) which results in the matching with the most Pareto improvements on the SOSM outcome, while another might choose the coalition(s) based on minimizing the preference index, and yet another could count the number of priority violations (the matching with fewer priority violations being more ``fair'') or weight the magnitude of the various priority violations (in terms of the priority level of the violator(s)).

We have already pointed out that in order for coalitions to work, there have to be some accomplices willing to modify their preferences despite the fact that they themselves will not benefit. While some of these accomplices might benefit from different coalitions, others, it turns out, have no hope of ever benefiting from any coalition. Huang calls these \textbf{hopeless men} in the context of the stable marriage problem, and proves that there always exists at least one hopeless man for any given set of preference profiles. We now consider the analogous construct for the School Choice Problem:
\begin{definition}
Given a specific SCP, a student who cannot benefit from any Pareto improvement upon SOSM is said to be {\bf hopeless}.
\end{definition}

It turns out that there is always such a hopeless student. In other words, we have:

\begin{theorem} 
There is always at least one hopeless student when Pareto improvements are made on SOSM. More specifically, in EADAM or CIM, there is always at least one hopeless student.
\end{theorem}

In fact this follows directly from a stronger result:

\begin{theorem}
\label{T:HopelessStudent}
The students who propose in the last round of SOSM are hopeless.
\end{theorem}

\begin{proof} Let $I$ be the set of all students and $S$ be the set of all schools. Label the set of schools that get proposals in the last step of SOSM as $S_L=\{s_{1},s_{2},\ldots,s_{m}\}$ and the set of students that propose in the last step as $I_L=\{i_{1},i_{2},\ldots,i_{n}\}$.
If a student $i_{j} \in I_L$ proposes to $s_{j} \in S_L$ in the last step of SOSM, then we can conclude that prior to that step, $s_j$ had not yet filled its capacity. Otherwise $i_j$ would be replacing another student at $s_j$ and thus forcing that student to propose in a next round, which would in turn mean that the algorithm could not stop.

Assume now that $i_j \in I_L$ is not hopeless. This means that he can improve his lot via some Pareto improvement.
Let us denote the set of ALL students who are moved with this improvement $I_P = \{i_j, i', i'', ... i^{(n)}\}$. If any of these students preferred vacant seats to their assigned seats at the end of SOSM we would have a contradiction; they should have included or ranked those schools higher on their preference lists.
Thus the Pareto improvement must move all these students to seats that have been assigned under SOSM. Thus the improvement in fact should correspond to a permutation of $I_P$, and by the pigeonhole principle there exists another student $i^*$ that prefers $s_{j}$ to her original assignment school $s^*$.
(It doesn't matter if $s^*$ is from $S_L$ or not. Similarly $i^*$ may or may not belong to $I_L$).  Since in SOSM students propose to schools in the order that they rank them on their preference lists, $i^*$ must have applied to $s_{j}$, been rejected, then applied to $s^*$. This would imply that school $s_{j}$ must have filled its capacity before the last step of SOSM, which 
contradicts our assumption that $i_j$ was accepted by $s_j$ in the last step of SOSM: no filled school can receive and accept a proposal in the final step of SOSM because that would result in a student being displaced which would require another step in the mechanism.
\end{proof}

Looking now back at $SCP_1$ we see that $i_3$ is a hopeless student. In other words, $i_3$ will not benefit from any Pareto improvement to the SOSM matching. Similarly in $SCP_2$, the last proposer in the SOSM is $i_5$ who is thus a hopeless student.

%%%%%%%%%%%%%%%%%%%
%%%%%%%%%%%%%%%%%%%
%%%%%%%%%%%%%%%%%%%

\subsection{Coalitions and EADAM}
\label{SS:CoalitionsEADAM}

The SOSM has already been implemented in several school districts due to its desirable properties of strategyproofness and stability and the fact that it generates the most (Pareto) efficient assignment among all stable matchings. However, as has already been mentioned (and documented in \cite{AbPaRo09, Ke10} and elsewhere) its strict adherence to a stable outcome (emphasizing priorities) can result in substantial inefficiency. In this section we compare two efficiency-oriented modifications to SOSM, namely the EADAM from \cite{Ke10} (described in \S\S\ref{SS:Background}) and the coalition cheating SOSM improvement.

The main result of this section is:

\begin{theorem}
There exists a coalition improvement on SOSM yielding the EADAM outcome with full consent.
\end{theorem}

More generally we will prove

\begin{theorem}
\label{T:CoalitionEADAM}
For any possible combination of consenters,  the associated EADAM outcome may be obtained by forming an appropriately designed coalition that improves on SOSM.
\end{theorem}

The intuition behind this is that ``accomplices" can be viewed as interrupters who consent to waive their priority so that they do not start a rejection chain. By each accomplice waiving his priority - those on the rejection chain are given the opportunity to be accepted into better schools on their lists without violating the priorities of the accomplices.

\begin{proof}[Proof of Theorem \ref{T:CoalitionEADAM}]
Let $I$ and $S$ be the sets of students and schools, respectively.   Let $({\bf P}, {\bf \Pi})$ be a given school choice problem for the pair $(I,S)$, and let $W$ be the set of students who consent to waiving their priorities under the EADAM mechanism. Denote by $M_S$ and $M_E$ the SOSM and the EADAM outcome matchings of this problem, respectively. We will now construct a coalition $C$ which will result in the same outcome $M_E$. First define the cabal set $K$ to be the set of all students whose assignments are different under $M_S$ and $M_E$:
\[ K = \{ i \in I : M_S[i] \neq M_E[i] \}. \]
These are the students who benefit from the EADAM; they will also be the students who will benefit from the coalition $C$. Since every student whose assignment changes under EADAM is in $K$, we can partition $K$ into cabal loops. This is equivalent to the basic algebraic fact that any finite permutation can be written as the product of disjoint cycles. Hence an elementary algorithm to decompose $K$ into its individual cabal loops can be described as follows:
\begin{itemize}
\item \underline{Step 0}: Define a permutation $\pi_K$ of $K$ by setting $\pi_K(i^{\prime}) = i$ ({\it $i^{\prime}$ points to $i$}) if $M_S[i] = M_E[i^{\prime}]$. In words, {\it $i^{\prime}$ points to $i$ if EADAM matches $i^{\prime}$ to the school that SOSM matches $i$}.
\item \underline{Step 1}: Pick a student $i \in K$ and label her $i_{1,1}$. Then let $i_{1,2}$ be the student $\pi_K(i_{1,1})$ and more generally label $i_{1,j+1} = \pi_K(i_{1,j})$. This process will stop at some $j_1$ with $\pi_K(i_{1,j_1}) = i_{1,1}$ as $\pi_K$ is a finite permutation. Then $K_1 = ( i_{1,1} \rightarrow i_{1,2} \rightarrow  \cdots \rightarrow i_{1,j_1} \rightarrow i_{1,1})$ is a cabal loop.
\item[] And in general,
\item \underline{Step k}, k $\ge$ 1: Pick a student $i \in K$ who has not yet been assigned to a cabal loop and label her $i_{k,1}$. If none exists then the algorithm stops. Otherwise, label $\pi_K(i_{k,1})$ as $i_{k,2}$ and more generally label $i_{k,j+1} = \pi_K(i_{k,j})$. This process stops at some $j_k$ with $\pi_K(i_{k,j_k}) = i_{k,1}$ as $\pi_K$ is finite. Then $K_k = (i_{k,1} \rightarrow i_{k,2} \rightarrow  \cdots \rightarrow i_{k,j_k} \rightarrow i_{k,1})$ is a cabal loop.
\end{itemize}
Note that the algorithm has to stop because $K$ is finite. Furthermore each student in $K$ shows up in exactly one round and hence in exactly one cabal loop, because $\pi_K$ is invertible.

Next we describe how to form the accomplice set $A(K)$. A student $i$ will be in $A(K)$ if and only if the following two conditions are both satisfied:
\begin{itemize}
\item $i \in W$, or equivalently, $i$ consents to waive her priorities in the EADAM; \textsc{and}
\item There is a school $s$ such that $(i,s)$ is a last interrupter pair at some round of EADAM.
\end{itemize}

The new preference profile for an accomplice $i \in A(K)$ will be of the form
\[ ({P_L[i]-X}, M_S[i], P_R[i]\cup X), \]
where
\begin{itemize}
\item if $i \not \in K$, then $X = \{s \in M_S[K] \vert s = M_S[i_k], s \succ_i M_S[i], i \succ_s i_{k+1} \}$, {\sc and}
\item if $i = i_k \in A(K) \cap K$, $X = \{s \in M_S[K] \vert s = M_S[i_j], j \neq k, s \succ_{i_k} M_S[i_{k}], i_k \succ_s i_{j+1}\}$.
\end{itemize}
Here we are using the notation of \S\S\ref{SS:CoalitionImplementation} where $P_L[i]$ (respectively $P_R[i]$) is the list of schools on $i$'s preference profile to the left (respectively to the right) of his stable assignment $M_S[i]$.

Finally Theorem \ref{T:SCPCoalition} allows us to conclude that the outcome matching $M_C$ of $C = (K,A(K))$ will be as follows: $M_C[i] = M_S[i]$ for all $i \not \in K$, and $M_C[i_k] = M_S[i_{k+1}]$ for $i_k, i_{k+1}$ in some cabal loop $K_j$ in $K$. But then $M_C = M_E$ and we are done.
\end{proof}

In order to see the analogy between accomplices and interruptors, we analyze a minor modification of $SCP_2$ from \S\S\ref{SS:CoalitionImplementation} which we label $SCP_3$: Let $I = \{i_1,i_2,i_3,i_4,i_5\}$ and $S=\{s_1,s_2,s_3,s_4,s_5\}$ be given with the following preference and priority structures, respectively:
\[ SCP_3 : \qquad
\begin{array}{ccc}
i_1:s_1 \succ s_2 \succ s_5 \succ s_4 \succ s_3 & \qquad & s_1:i_3 \succ i_2 \succ i_4 \succ i_1 \succ i_5\\
i_2:s_2 \succ s_5 \succ s_4 \succ s_1 \succ s_3 && s_2:i_4 \succ i_5 \succ i_1 \succ i_2 \succ i_3 \\
i_3:s_5 \succ s_2 \succ s_1 \succ s_3 \succ s_4 && s_3:i_2 \succ i_3 \succ i_4 \succ i_5 \succ i_1\\
i_4:s_4 \succ s_1 \succ s_2 \succ s_3 \succ s_5 && s_4:i_1 \succ i_2 \succ i_3 \succ i_5 \succ i_4\\
i_5:s_5 \succ s_4 \succ s_2 \succ s_3 \succ s_1 && s_5:i_1 \succ i_2 \succ i_5 \succ i_3 \succ i_4
\end{array} \]
We now run the SOSM algorithm for $SCP_3$ (this is also Round 1 for EADAM assuming full consent):
\begin{center}
\small
  \begin{tabular}{ | l || c | c | c | c | c |}
    \hline
   Round 1 & $s_1$ & $s_2$ & $s_3$ & $s_4$ & $s_5$ \\ \hline
    \hline
    Step 1 & $i_1$ & $i_2$ & & $i_4$ & \sout{$i_3$}, $i_5$ \\ \hline
    Step 2 &   & $i_2$, \sout{$i_3$} & & &  \\ \hline
    Step 3 & $i_3$, \sout{$i_1$} &  & & &  \\ \hline
    Step 4 &  & $i_1$, \sout{$i_2$} & &  &  \\ \hline
    Step 5 &  & & &  & $i_2$, \sout{$i_5$} \\ \hline
    Step 6 &  & & & \sout{$i_4$}, $i_5$ &  \\ \hline
    Step 7 & $i_3$, \sout{$i_4$} & & & &  \\ \hline
    Step 8 &  & \sout{$i_1$}, $i_4$ & &  &  \\ \hline
    Step 9 &  &  & &  & $i_1$, \sout{$i_2$} \\ \hline
    Step 10 &  &  & & $i_2$, \sout{$i_5$} &  \\ \hline
    Step 11 &  & $i_4$, \sout{$i_5$} & &  &  \\ \hline
    Step 12 &  &  & $i_5$ & &  \\ \hline
    \hline
  \end{tabular}
  \normalsize
  \end{center}
The SOSM outcome is:
\[ M_S^{SCP_3} =
\begin{pmatrix}i_1&i_2&i_3&i_4&i_5 \\ s_5&s_4&s_1&s_2&s_3 \end{pmatrix}.
\]
Note that student $i_1$ is an interrupter for $s_2$ (causes $s_2$ to reject $i_2$ in step 4 and is rejected herself by $s_2$ when $i_4$ comes along, in step 8.) Student $i_2$ is an interrupter for $s_2$ (causes $s_2$ to reject $i_3$ in step 2 and is rejected herself by $s_2$ when $i_1$ comes along, in step 4.) Student $i_5$ is an interrupter for $s_4$ (causes $s_4$ to reject $i_4$ in step 6 and is rejected herself by $s_4$ when $i_2$ comes along, in step 10.) Student $i_5$ is an interrupter for $s_5$ (causes $s_5$ to reject $i_3$ in step 1 and is rejected herself by $s_5$ when $i_2$ comes along, in step 5.) Student $i_2$ is an interrupter for $s_5$ (causes $s_5$ to reject $i_5$ in step 5 and is rejected herself by $s_5$ when $i_1$ comes along, in step 9.) Thus $\{i_5,s_4\}$ is the last interrupter pair. We remove $s_4$ from $i_5$'s preference list.\footnote{In each round of EADAM, there may be multiple interrupter pairs. However, only the LATEST interrupter pair is used for the next round. For example, in Round 1, there are 5 interrupter pairs, but we only remove the latest interrupter pair $\{i_5, s_4\}$.}
\begin{center}
\small
  \begin{tabular}{ | l || c | c | c | c | c |}
    \hline
   Round 2 & $s_1$ & $s_2$ & $s_3$ & $s_4$ & $s_5$ \\ \hline
    \hline
    Step 1 & $i_1$ & $i_2$ & & $i_4$ & \sout{$i_3$}, $i_5$ \\ \hline
    Step 2 &   & $i_2$, \sout{$i_3$} & & &  \\ \hline
    Step 3 & $i_3$, \sout{$i_1$} &  & & &  \\ \hline
    Step 4 &  & $i_1$, \sout{$i_2$} & &  &  \\ \hline
    Step 5 &  & & &  & $i_2$, \sout{$i_5$} \\ \hline
    Step 6 &  & \sout{$i_1$}, $i_5$ & &  &  \\ \hline
    Step 7 &  & & & & $i_1$, \sout{$i_2$} \\ \hline
    Step 8 &  &  & & $i_2$, \sout{$i_4$} &  \\ \hline
    Step 9 &  $i_3$, \sout{$i_4$} &  & &  &  \\ \hline
    Step 10 &  & $i_4$, \sout{$i_5$} & &  &  \\ \hline
    Step 11 &  &  & $i_5$ &  &  \\ \hline
    \hline
  \end{tabular}
  \normalsize
\end{center}
The outcome is the same in Round 2. The pair $\{i_5,s_2\}$ is the last interrupter pair. We remove $s_2$ from $i_5$'s preference list.
\begin{center}
\small
  \begin{tabular}{ | l || c | c | c | c | c |}
    \hline
   Round 3 & $s_1$ & $s_2$ & $s_3$ & $s_4$ & $s_5$ \\ \hline
    \hline
    Step 1 & $i_1$ & $i_2$ & & $i_4$ & \sout{$i_3$}, $i_5$ \\ \hline
    Step 2 &   & $i_2$, \sout{$i_3$} & & &  \\ \hline
    Step 3 & $i_3$, \sout{$i_1$} &  & & &  \\ \hline
    Step 4 &  & $i_1$, \sout{$i_2$} & &  &  \\ \hline
    Step 5 &  & & &  & $i_2$, \sout{$i_5$} \\ \hline
    Step 6  &  &  & $i_5$ &  &  \\ \hline
    \hline
  \end{tabular}
  \normalsize
  \end{center}
The outcome is different in Round 3; $i_1$, $i_2$ and $i_4$'s assignments have changed. This time $\{i_5,s_5\}$ is the last interrupter pair. We remove $s_5$ from $i_5$'s preference list.
\begin{center}
\small
  \begin{tabular}{ | l || c | c | c | c | c |}
    \hline
   Round 3 & $s_1$ & $s_2$ & $s_3$ & $s_4$ & $s_5$ \\ \hline
    \hline
    Step 1 & $i_1$ & $i_2$ & $i_5$ & $i_4$ & $i_3$ \\ \hline
    \hline
  \end{tabular}
  \normalsize
  \end{center}
The last round (Round 4) takes only one step; everybody gets matched with his or her first choice. (Of course $i_5$ has had to make adjustments to his profile multiple times; his new first choice was in fact his fourth truthful choice.)
Thus EADAM with full consent\footnote{In fact we only used the consent of $i_5$.} returns the following matching:
\[ M_E^{SCP_3} =
\begin{pmatrix}i_1&i_2&i_3&i_4&i_5 \\ s_1&s_2&s_5&s_4&s_3 \end{pmatrix}.
\]
Can we find a coalition that will output this same outcome? Indeed yes! The cabal will be the set $\{i_1,i_2,i_3,i_4\}$ and the singleton accomplice will be $\{i_5\}$. The set $X$ for $i_5$ will be $X =\{s_2,s_4,s_5\}$.
Note that there are two cabal loops: $(i_1 \rightarrow i_3 \rightarrow i_1)$ and $(i_2 \rightarrow i_4 \rightarrow i_2)$.

Note that the coalition we create (and equivalently the EADAM with full consent) in this problem corresponds to a drastic improvement in the preference index. The preference index of this outcome is 3 while the preference index of the original SOSM outcome was 11.

There are indeed other coalitions that could be used for the same SCP. Take for instance the cabal to be $\{i_2,i_4\}$ and let $\{i_5\}$ to be the singleton accomplice set. Then $X = \{s_2,s_4\}$ and we get:
\[ M_C^{SCP_3} =
\begin{pmatrix}i_1&i_2&i_3&i_4&i_5 \\ s_5&s_2&s_1&s_4&s_3 \end{pmatrix}.
\]
The preference index for this matching is 7. Thus it is still an improvement by this measure on the SOSM, but the larger cabal of the previous example (which is equivalent to EADAM with $i_5$'s consent) is more optimal with respect to reducing the preference index. On the other hand it may be interesting to observe that this outcome cannot be obtained via EADAM no matter which students consent. This is due to the fact that once $i_5$ consents to waive his priorities, he has to consent fully. In other words, any other Pareto improvement involving the interrupter pairs he was a part of will also be made. This in particular implies that the converse of Theorem \ref{T:CoalitionEADAM} is not true.

Looking more closely at $SCP_3$, we notice that in practice what we have done amounts to changing the preferences for $i_5$ to
$i_5 : s_3 \succ s_4 \succ s_2 \succ s_5 \succ s_1$; this way we were able to give $i_1, i_2, i_3, i_4$ their first choice without making $i_5$ worse off. However, we can alternatively change the preferences for $i_3$ and give $i_1, i_2, i_4, i_5$ their first choice and $i_3$ her 4th choice which is essentially the same matching (also with preference index 3!) except that $i_3$ is worse off now than in the SOSM matching.

Student $i_5$ seems to be set to lose out from the beginning, even in the SOSM matching. We may ask whether this is due to the fact that $i_5$, unlike $i_3$, is not highly prioritized at any school, and how the fact that $i_5$ is a hopeless student relates to this situation. 
The crucial point is that the matching obtained via the coalition formed by changing $i_5$'s preferences, and including everyone else in the cabal, or equivalently the EADAM outcome, (Pareto) dominates the original SOSM, while the matching obtained when $i_3$ is made to change his preference list does not. Changing $i_3$'s preferences would be more objectionable than changing $i_5$'s preferences because the associated matching indeed harms somebody when compared to their stable assignment under SOSM, which is viewed as a baseline assigning an initial endowment to each student.

%%%%%%%%%%%%%%%%%%%
%%%%%%%%%%%%%%%%%%%
%%%%%%%%%%%%%%%%%%%

\subsection{Another efficiency oriented mechanism: TTC and coalitions}
\label{SS:CoalitionsTTC}

We now look for a relationship between the Top Trading Cycles mechanism and SOSM with the coalition efficiency improvement. When we incorporate coalitions into the SOSM, we  alter preferences so that some students have improved placements and no students have worse placements (relative to their standard assignments under SOSM). We now seek to understand if we can mimic the outcome of TTC via some coalition/preference manipulation of SOSM.

We begin with an example in which the TTC and SOSM find two different matchings, yet a coalition can be formed so that SOSM with this coalition results in the same matching as the TTC. Let us look once again at $SCP_1$ from \S\ref{S:Introduction}: 
There are three schools, $s_{1}, s_{2}, s_{3}$, three students $i_{1}, i_{2}, i_{3}$, and only one seat at each school. Student preferences and school priorities are given by:
\[ SCP_1 : \qquad \begin{array}{ccc}
i_{1}:s_{2} \succ s_{1} \succ s_{3} & \qquad & s_{1}: i_{1} \succ i_{3} \succ i_{2} \\
i_{2}: s_{1}  \succ  s_{2} \succ s_{3} & \qquad & s_{2}:i_{2} \succ i_{1} \succ i_{3} \\
i_{3}: s_{1} \succ s_{2} \succ s_{3} & \qquad &  s_{3}:i_{2} \succ i_{1} \succ i_{3}
\end{array}\]
Here, the SOSM finds the matching (of preference index $4$):
\[ M_S^{SCP_1} =
\begin{pmatrix}i_{1}& i_{2}& i_{3}\\ s_{1}&s_{2}&s_{3}\end{pmatrix}.
\]
The TTC finds a matching that (Pareto) dominates $M_S^{SCP_1}$ (and has preference index $2$):\footnote{Recall that this is the same as the EADAM outcome in the case of full consent: $M_T^{SCP_1}  = M_E^{SCP_1}$.}
\[ M_T^{SCP_1}  =
\begin{pmatrix}i_{1}& i_{2}& i_{3}\\ s_{2}&s_{1}&s_{3}\end{pmatrix}.
\]
If $i_{3}$ modifies her preference list such that $s_{3}$ is her first choice, then each student has a distinct first choice. Then SOSM assigns all their first choice, which yields the same matching as the TTC output for the original setup. The relevant coalition is given by $K = \{i_{1}, i_{2}\}$ and $A(K) = \{i_{3}\}$.

Similarly if we run the TTC algorithm on $SCP_2$, we see that the outcome is the same as that obtained when we use the coalition with the cabal $K = \{i_1,i_2,i_4\}$ (the cabal loop is $(i_1 \rightarrow i_4 \rightarrow i_2 \rightarrow i_1)$), the accomplice set $A(K) = \{i_5\}$, and the set $X = \{s_2,s_4\}$ for $i_5$. In other words the TTC outcome is equivalent to the coalition-adjusted outcome of SOSM with $i_5$ being the only person who needs to modify her preference profile (\S\S\ref{SS:CoalitionImplementation}, also see \S\S\ref{SS:CliquesDescription}).

With these examples as background we now prove:

\begin{theorem}
\label{T:CoalitionAdjustments}
One cannot always obtain the TTC outcome by coalition adjustments to SOSM.
\end{theorem}

\begin{proof}
We construct a counterexample. We modify $SCP_1$ slightly and call the new setup $SCP_4$.
Assume there are three schools, $s_{1}, s_{2}, s_{3}$, three students $i_{1}, i_{2}, i_{3}$, and there are two seats at $s_{2}$ and one seat each at $s_{1}$ and $s_{3}$. School priorities and student preferences are as in $SCP_1$:
\[ SCP_4 : \qquad \begin{array}{ccc}
i_{1}: s_{2} \succ s_{1} \succ s_{3} & \qquad & s_{1}: i_{1} \succ i_{3} \succ i_{2} \\
i_{2}: s_{1}  \succ  s_{2} \succ s_{3} & \qquad & s_{2}:i_{2} \succ i_{1} \succ i_{3} \\
i_{3}: s_{1} \succ s_{2} \succ s_{3} & \qquad &  s_{3}:i_{2} \succ i_{1} \succ i_{3}
\end{array}\]
Here the SOSM matching is:
\[ M_S^{SCP_4} =
\begin{pmatrix}i_{1}& i_{2}& i_{3}\\ s_{2}&s_{2}&s_{1}\end{pmatrix},\]
and it has preference index $1$.

If we run the TTC for $SCP_4$, we find one cycle in the first round: $i_{1}$ points to $s_{2}$, $s_{2}$ points to $i_{2}$, $i_{2}$ points to $s_{1}$, and $s_{1}$ points back to $i_{1}$. We assign these students as desired. In the next round, $i_{3}$ is the only student left, and the unfilled schools are $s_{2}$ and $s_{3}$. $i_{3}$ points to $s_{2}$ and $s_{2}$ points back. This is a cycle, so we assign $i_{3}$ to $s_{2}$. Thus, our final matching is:
\[ M_T^{SCP_4} =
\begin{pmatrix}i_{1}& i_{2}& i_{3}\\ s_{2}&s_{1}&s_{2}\end{pmatrix},\]
which also has preference index $1$.

The two matchings are Pareto incomparable.
As coalition improvements are Pareto improvements, it is impossible to design a coalition so that the SOSM finds the same matching as the TTC.\footnote{In this particular problem, the SOSM performs undeniably better against the criteria we use to evaluate outcomes. Its outcome has low preference index, is stable and efficient, while the TTC outcome is not stable.}
\end{proof}

In the example above the SOSM outcome was efficient. However, there are cases when the SOSM outcome is not efficient, where the TTC outcome is Pareto incomparable and so cannot be obtained via a coalition improvement. Consider for instance the following school choice problem we label $SCP_5$: Let $I=\{i_1, i_2, i_3, i_4\}$ and $S=\{s_1, s_2, s_3, s_4\}$, and assume that each school has only one seat. School priorities and student preferences are given as follows:
\[ SCP_5 : \qquad
\begin{array}{ccc}
i_1: s_1 \succ s_3 \succ s_2 \succ s_4  &  \qquad  &  s_1: i_4 \succ i_3 \succ i_1 \succ i_2\\
i_2: s_1 \succ s_2 \succ s_3 \succ s_4  &  \qquad  &  s_2: i_2 \succ i_3 \succ i_1 \succ i_4\\
i_3: s_2 \succ s_1 \succ s_3 \succ s_4  & \qquad  &  s_3: i_2 \succ i_3 \succ i_1 \succ i_4\\
i_4: s_4 \succ s_3 \succ s_2 \succ s_1  &  \qquad  &   s_4: i_1 \succ i_2 \succ i_3 \succ i_4 \end{array} \]
If we run both mechanisms, we see that the respective assignments are:
\[ M_S^{SCP_5} =
\begin{pmatrix}i_1&i_2&i_3&i_4\\s_3&s_2&s_1&s_4 \end{pmatrix}
\qquad \qquad M_T^{SCP_5} = \begin{pmatrix}i_1&i_2&i_3&i_4\\s_1&s_2&s_3&s_4\end{pmatrix}\]
with preference indices both equaling $3$.
Looking at the two assignments, we see that the matching under SOSM is not efficient, and TTC and SOSM are not Pareto comparable. 
\footnote{An efficient Pareto improvement of $M_S^{SCP_5}$ with preference index $1$ can be obtained via EADAM with the consent of $i_1$, or equivalently, via a coalition with $K = \{i_2, i_3\}$, $A(K) = \{i_1\}$, and $X=\{s_1\}$ for $i_1$.}

The above discussion allows us to conclude that in case the TTC outcome  dominates the SOSM outcome, we can obtain the same outcome via an appropriate choice of coalitions (e.g.\ $SCP_1$). However when the TTC outcome is Pareto incomparable to the SOSM outcome (e.g.\ $SCP_4$, $SCP_5$), it is impossible to obtain the same outcome via a coalition adjustment.
More generally we can prove:

\begin{theorem}
The outcome of a mechanism that is Pareto incomparable to SOSM cannot be obtained by a coalition adjustment to SOSM.
\end{theorem}

This follows from the fact that coalition adjustments result in Pareto improvements.

%%%%%%%%%%%%%%%%%%%
%%%%%%%%%%%%%%%%%%%
%%%%%%%%%%%%%%%%%%%

\section{Trading Cliques for school choice}
\label{S:Cliques}

EADAM and CIM provide us with ways to systematically improve upon the SOSM matching. However, both mechanisms involve complicated procedures. For the EADAM, we need to identify interruptors in a backward order one by one and run the SOSM algorithm over and over again. For CIM, we need to form various coalitions and identify falsified preference orders that will work for each. However the ultimate goal in either case is the same: 
to Pareto improve upon SOSM while being able to justify the resulting priority violation(s).

In this section we propose another way to improve upon the SOSM matching. Although there will be priority violations in the final matching (since SOSM Pareto dominates all stable matchings), we justify these violations by noting that nobody is made worse off in their re-assignment when compared to their SOSM assignment. School priorities are respected by the SOSM matching, so there is also a baseline reverence to these constraints, though they are deemphasized in later stages.

The main idea is as follows: We begin by applying SOSM to the given student preferences and school priorities. Next with no further consideration of school priorities, we enter students into a trading market. In other words, the SOSM assignment is the starting point for the next phase of the assignment process after which priorities are ignored. For the second round, the goal is to improve school assignments from the point of view of student preferences as much as is possible.

In \S\S\ref{SS:CliquesDescription}, we describe in more detail our new mechanism, the {\it Trading Adjusted Deferred Acceptance Mechanism} (TADAM). We investigate the basic properties of TADAM and compare the outcome of TADAM with those of other standard mechanisms in \S\S\ref{SS:CliquesProperties}. In particular we discuss how coalitions and cliques relate to one another and to other mechanisms involving cycle improvements. We also comment on the implications for the school choice context.

%%%%%%%%%%%%%%%%%%%
%%%%%%%%%%%%%%%%%%%
%%%%%%%%%%%%%%%%%%%

\subsection{Description of the Trading Adjusted Deferred Acceptance Mechanism}
\label{SS:CliquesDescription}

We now develop a systematic way to find all Pareto improvements upon a predetermined matching $M$ in a given SCP. We start by associating a directed weighted graph $(V,E,w)$ to $M$ as follows: Each student $i$ is assigned a unique vertex $v_i$ in $V$. There is an edge from vertex $v_i$ to vertex $v_j$ if student $i$ desires student $j$'s assignment under the given matching at least as much as, if not more than, the school to which he himself was assigned. An edge $e$ from vertex $v_i$ to vertex $v_j$ has weight $w(e) = 0$ if student $i$ desires student $j$'s assignment under the given matching as much as, but not more than, the school to which he himself was assigned, and $w(e) = 1$ if the preference is strict. 

In the above we can identify $V$ with the set of students. With this in mind we now introduce:

\begin{definition}
\label{D:DirectedGraph}
Let $I$ and $S$ be a set of $n$ students and a set of $m$ schools, respectively, with respective preference and priority structures $({\bf P},{\bf \Pi})$. Let $M$ be a matching for the associated SCP. We say that the directed weighted graph $G_M = (V,E,w)$ is the {\bf (directed weighted) graph of the matching $M$} if $V = I$; for any pair of students $(i,j)$, there is an edge $e_{ij}$ from $i$ to $j$ if and only if $M[j] \succeq_i M[i]$; and for each edge $e_{ij} \in E$, $w(e_{ij}) = 0$ if $M[i] \succeq_i M[j]$, and  
$w(e_{ij}) = 1$ otherwise.
\end{definition}

Using this terminology, we can make the following:

\begin{definition}(cf.\ \cite[Defn.\ 1]{Er02})
Let $I$, $S$, $({\bf P},{\bf \Pi})$, $M$ and $G_M$ be given as in Defn.\ \ref{D:DirectedGraph} and let $k \in \N$. A {\bf trading clique of length $k$} (or simply a {\bf clique}) consists of a sequence $(i_1,i_2, \cdots, i_k)$ of $k$ distinct students such that for each $s < k$, there is an edge in $E$ from $v_{i_s}$ to $v_{i_{s+1}}$, there is an edge in $E$ connecting $v_{i_k}$ back to $v_{i_1}$, and either for some $s < k$, $w(e_{i_s,i_{s+1}}) = 1$ or $w(e_{i_k,i_1}) = 1$. A similar cycle where $w = 0$ on all edges is called a {\bf null clique}. A matching whose graph contains no trading cliques is {\bf acyclical}.
\end{definition}

A straightforward result then follows:
\begin{theorem}\label{T:paretocycle}
If matching $M$ (Pareto) dominates matching $M^{\prime}$, then the directed graph $G_{M^{\prime}}$ of $M^{\prime}$ admits a trading clique. Equivalently, if the directed graph of $M^{\prime}$ is acyclical, then $M^{\prime}$ is Pareto efficient. Conversely, if $M^{\prime}$ admits a trading clique, we can always find a matching $M$ which Pareto dominates $M^{\prime}$  (equivalently, the directed graph of a Pareto efficient matching is acyclical). 
\end{theorem}

\begin{proof}
Let $M$ and $M^{\prime}$ be two matchings such that $M$ (Pareto) dominates $M^{\prime}$ but the directed graph of $M^{\prime}$ has no trading cliques. As $M$ dominates $M^{\prime}$, some students get strictly better off by changing from $M^{\prime}$ to $M$. We draw a directed edge from each improved student to the student who brought that school into the trade as his initial endowment. Since we assume that $G_{M^{\prime}}$ does not admit a clique, these edges cannot constitute a cycle. Hence there exists at least one student at the end of a chain of such directed edges. This student then only gives out his endowment but is not receiving any (better) school. This contradicts with the assumption of Pareto domination.
\end{proof}

Consider now the following procedure:

\begin{itemize}
\item \underline{Round 0}: Given a preference and priority profile, run the SOSM algorithm and obtain a temporary matching $M_0$.
\item \underline{Round t}, t $\ge$ 1: Given $M_{t-1}$, consider the graph $(V_t, E_t, w_t)$ of $M_{t-1}$. If there exists a student with no path through him, remove that student from the graph; his assignment under $M_t$ will be his initial endowment at the beginning of this round. If there are any trading cliques in the graph $(V_t, E_t)$, pick one. For each edge from $i$ to $j$ in this clique, let $M_t$ be the matching that assigns student $i$ the school to which $j$ was matched under $M_{t-1}$. If there is no trading clique, return $M_{t-1}$ as the outcome $M_t$ and stop.
\end{itemize}

We call this the {\bf Trading Adjusted Deferred Acceptance Mechanism} (TADAM). Note that adjusting student assignments by following a trading clique yields a Pareto improvement. Thus all subsequent outcomes Pareto dominate the SOSM matching. Looking also at when the algorithm stops we can in fact qualify the last assertion further and say that all outcomes of the TADAM are Pareto efficient Pareto dominations of the initial SOSM matching.

As can be seen from the steps of the defining algorithm above, there may be several outcomes of TADAM for a given problem. In particular in cases with multiple trading cliques the process may output different matchings depending on which cycles are selected at rounds $t \ge 1$. 
This will be clearer when we look at concrete examples, which we do next.

We begin with an example where the preference and priority structures are strict.\footnote{In such a situation, the weight function on the graph is uniformly $1$ and can be ignored.} Consider once again $SCP_2$ (\S\S\ref{SS:CoalitionImplementation}) with five students and five schools each with one seat:
\[SCP_2 : \qquad \begin{array}{ccc}
i_1: s_2 \succ s_5 \succ s_4 \succ s_3 \succ s_1 & \qquad  & s_1:i_3 \succ i_2 \succ i_4 \succ i_1 \succ i_5\\
i_2: s_2 \succ s_5 \succ s_4 \succ s_1 \succ s_3 & \qquad  &  s_2: i_4  \succ  i_5\succ i_1 \succ i_2 \succ i_3\\
i_3: s_5 \succ s_2 \succ s_1 \succ s_3 \succ s_4 & \qquad  & s_3: i_2 \succ i_3 \succ i_4 \succ i_5 \succ i_1\\
i_4: s_4 \succ s_1 \succ s_2 \succ s_3 \succ s_5 & \qquad  & s_4: i_1 \succ i_2 \succ i_3 \succ i_5 \succ i_4\\
i_5: s_5 \succ s_4 \succ s_2 \succ s_3 \succ s_1 & \qquad  & s_5: i_1 \succ i_2 \succ i_5 \succ i_3 \succ i_4
\end{array}\]
The matching under SOSM is:
\[M_S^{SCP_2} = \begin{pmatrix}
i_1&i_2&i_3&i_4&i_5 \\
s_5&s_4&s_1&s_2&s_3
\end{pmatrix},\]
and has preference index $10$. SOSM does a poor job with student preferences here. One student gets his fourth choice, three get their third choice and one gets his second choice.

For $SCP_2$, the associated SOSM matching can thus be translated into the following graph.
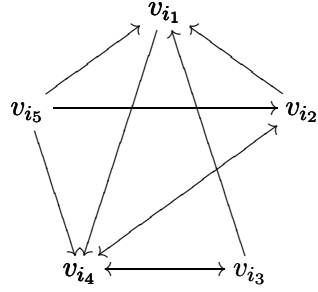
\[ \bfig
\node v_{i_1}(0,500)[v_{i_1}]
\node v_{i_5}(-475,154)[v_{i_5}]
\node v_{i_2}(475,154)[v_{i_2}]
\node v_{i_3}(294,-404)[v_{i_3}]
\node v_{i_4}(-294,-404)[v_{i_4}]
\arrow[v_{i_1}`v_{i_4};]
\arrow/<->/[v_{i_2}`v_{i_4};]
\arrow[v_{i_2}`v_{i_1};]
\arrow/<->/[v_{i_3}`v_{i_4};]
\arrow[v_{i_3}`v_{i_1};]
\arrow[v_{i_5}`v_{i_1};]
\arrow[v_{i_5}`v_{i_2};]
\arrow[v_{i_5}`v_{i_4};]
\efig \]

We see that if there is an arrow from $i_l$ to $i_j$ then $i_l$ would (weakly) prefer to be assigned to $M[i_j]$.  Such a swap can only be allowed if another student, $i_k$, prefers $M[i_l]$ to his own assignment, that is, only if there is a directed edge from some $v_{i_k}$ to $v_{i_l}$.  In this manner, a group of students can form a ``swap market'' and they can trade their SOSM assignments among themselves consistent with the directed graph.  Such a ``swap market'' would correspond to a cycle in the graph.  Here are four different  trading cliques
within the directed graph above (cliques denoted by unbroken arrows):
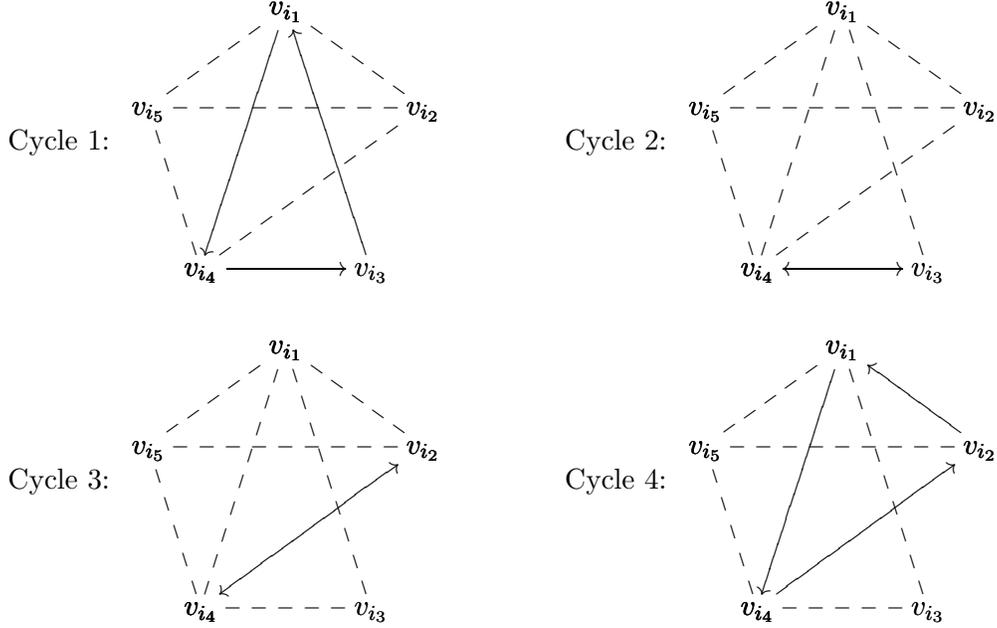
\[ \text{Cycle 1: } \bfig
\node v_{i_1}(-250,500)[v_{i_1}]
\node v_{i_5}(-725,154)[v_{i_5}]
\node v_{i_2}(225,154)[v_{i_2}]
\node v_{i_3}(44,-404)[v_{i_3}]
\node v_{i_4}(-544,-404)[v_{i_4}]
\arrow[v_{i_1}`v_{i_4};]
\arrow/--/[v_{i_2}`v_{i_4};]
\arrow/--/[v_{i_2}`v_{i_1};]
\arrow[v_{i_3}`v_{i_1};]
\arrow[v_{i_4}`v_{i_3};]
\arrow/--/[v_{i_5}`v_{i_1};]
\arrow/--/[v_{i_5}`v_{i_2};]
\arrow/--/[v_{i_5}`v_{i_4};]
\efig \qquad \qquad
 \text{Cycle 2: } \bfig
\node v_{i_1}(250,500)[v_{i_1}]
\node v_{i_5}(-225,154)[v_{i_5}]
\node v_{i_2}(725,154)[v_{i_2}]
\node v_{i_3}(544,-404)[v_{i_3}]
\node v_{i_4}(-44,-404)[v_{i_4}]
\arrow/--/[v_{i_1}`v_{i_4};]
\arrow/--/[v_{i_2}`v_{i_4};]
\arrow/--/[v_{i_2}`v_{i_1};]
\arrow/--/[v_{i_3}`v_{i_1};]
\arrow/<->/[v_{i_4}`v_{i_3};]
\arrow/--/[v_{i_5}`v_{i_1};]
\arrow/--/[v_{i_5}`v_{i_2};]
\arrow/--/[v_{i_5}`v_{i_4};]
\efig
\]

\[ \text{Cycle 3: } \bfig
\node v_{i_1}(-250,500)[v_{i_1}]
\node v_{i_5}(-725,154)[v_{i_5}]
\node v_{i_2}(225,154)[v_{i_2}]
\node v_{i_3}(44,-404)[v_{i_3}]
\node v_{i_4}(-544,-404)[v_{i_4}]
\arrow/--/[v_{i_1}`v_{i_4};]
\arrow/<->/[v_{i_2}`v_{i_4};]
\arrow/--/[v_{i_2}`v_{i_1};]
\arrow/--/[v_{i_3}`v_{i_4};]
\arrow/--/[v_{i_3}`v_{i_1};]
\arrow/--/[v_{i_5}`v_{i_1};]
\arrow/--/[v_{i_5}`v_{i_2};]
\arrow/--/[v_{i_5}`v_{i_4};]
\efig \qquad \qquad
 \text{Cycle 4: } \bfig
\node v_{i_1}(250,500)[v_{i_1}]
\node v_{i_5}(-225,154)[v_{i_5}]
\node v_{i_2}(725,154)[v_{i_2}]
\node v_{i_3}(544,-404)[v_{i_3}]
\node v_{i_4}(-44,-404)[v_{i_4}]
\arrow[v_{i_1}`v_{i_4};]
\arrow[v_{i_2}`v_{i_1};]
\arrow/--/[v_{i_3}`v_{i_4};]
\arrow/--/[v_{i_3}`v_{i_1};]
\arrow[v_{i_4}`v_{i_2};]
\arrow/--/[v_{i_5}`v_{i_1};]
\arrow/--/[v_{i_5}`v_{i_2};]
\arrow/--/[v_{i_5}`v_{i_4};]
\efig
\]

 We list the assignments corresponding to each of the four cliques:
 \[ M_1 =
\begin{array}{ccc}
i_1: \underline{s_2} \succ s_5 \succ s_4 \succ s_3 \succ s_1 \\
i_2:s_2 \succ s_5 \succ \underline{s_4} \succ s_1 \succ s_3 \\
i_3:\underline{s_5} \succ s_2 \succ s_1 \succ s_3 \succ s_4 \\
i_4:s_4 \succ \underline{s_1}\succ s_2 \succ s_3 \succ s_5 \\
i_5:s_5 \succ s_4 \succ s_2 \succ \underline{s_3} \succ s_1
\end{array} \qquad
M_2 = \begin{array}{ccc}
i_1:  s_2 \succ \underline{s_5}\succ s_4 \succ s_3 \succ s_1  \\
i_2: s_2 \succ s_5 \succ \underline{s_4} \succ s_1 \succ s_3 \\
i_3: :s_5 \succ \underline{s_2} \succ s_1 \succ s_3 \succ s_4 \\
i_4: s_4 \succ \underline{s_1} \succ s_2 \succ s_3 \succ s_5 \\
i_5: s_5 \succ s_4 \succ s_2 \succ \underline{s_3}\succ s_1
\end{array}\]
\[M_3 = \begin{array}{ccc}
i_1: s_2 \succ \underline{s_5}\succ s_4 \succ s_3 \succ s_1 \\
i_2:\underline{s_2} \succ s_5 \succ s_4 \succ s_1 \succ s_3 \\
i_3:s_5 \succ s_2 \succ \underline{s_1} \succ s_3 \succ s_4 \\
i_4:\underline{s_4}\succ s_1 \succ s_2 \succ s_3 \succ s_5 \\
i_5:s_5 \succ s_4 \succ s_2 \succ \underline{s_3} \succ s_1
\end{array} \qquad
M_4 = \begin{array}{ccc}
i_1: \underline{s_2} \succ s_5 \succ s_4 \succ s_3 \succ s_1 \\
i_2: s_2 \succ \underline{s_5} \succ s_4 \succ s_1 \succ s_3 \\
i_3: s_5 \succ s_2 \succ \underline{s_1} \succ s_3 \succ s_4 \\
i_4: \underline{s_4} \succ s_1 \succ s_2 \succ s_3 \succ s_5 \\
i_5: s_5 \succ s_4 \succ s_2 \succ \underline{s_3}\succ s_1
\end{array}\]

Observe that matchings 1, 3, and 4 are Pareto efficient but matching 2 is not. In fact, if we draw the directed graph of matching 2, we see that there is another cycle between $i_3$ and $i_1$. Thus we could continue with another clique, which would result in matching 1.

The preference reverence index for all three efficient matchings is 6. This raises the question of what efficient matching should be chosen in case of multiple efficient matchings. In this specific example, all three matchings give two students their top choice, one student her second choice, one student her third choice, and one student her fourth choice. Note that matching 4 is the one obtained earlier via EADAM with the consent of $i_5$ and, equivalently, via a coalition with the cabal $K = \{i_1,i_2,i_4\}$ (the cabal loop is $(i_1 \rightarrow i_4 \rightarrow i_2 \rightarrow i_1)$), the accomplice set $A(K) = \{i_5\}$, and the set $X = \{s_2,s_4\}$ for $i_5$  (cf.\ \S\S\ref{SS:CoalitionImplementation}). We also saw in \S\S\ref{SS:CoalitionsTTC} that this is exactly the TTC outcome.

Note that in all these cases, $i_5$'s assignment stays the same, i.e., $i_5$ is a hopeless student a la \S\S\ref{SS:Hopeless}. Looking at the graph, we see that there is no path passing through $i_5$; there is no chance for his situation to be improved. We can simplify the graph by taking out the vertex corresponding to $i_5$.

Let us look now at $SCP_3$ (\S\ref{SS:CoalitionsEADAM}) and see what TADAM would yield in that situation. Recall that the preference and priority profiles were as follows::
\[ SCP_3 : \qquad
\begin{array}{ccc}
i_1:s_1 \succ s_2 \succ s_5 \succ s_4 \succ s_3 & \qquad & s_1:i_3 \succ i_2 \succ i_4 \succ i_1 \succ i_5\\
i_2:s_2 \succ s_5 \succ s_4 \succ s_1 \succ s_3 && s_2:i_4 \succ i_5 \succ i_1 \succ i_2 \succ i_3 \\
i_3:s_5 \succ s_2 \succ s_1 \succ s_3 \succ s_4 && s_3:i_2 \succ i_3 \succ i_4 \succ i_5 \succ i_1\\
i_4:s_4 \succ s_1 \succ s_2 \succ s_3 \succ s_5 && s_4:i_1 \succ i_2 \succ i_3 \succ i_5 \succ i_4\\
i_5:s_5 \succ s_4 \succ s_2 \succ s_3 \succ s_1 && s_5:i_1 \succ i_2 \succ i_5 \succ i_3 \succ i_4
\end{array} \]
The SOSM matching for this problem was:
\[ M_S^{SCP_3} =
\begin{pmatrix}i_1&i_2&i_3&i_4&i_5 \\ s_5&s_4&s_1&s_2&s_3 \end{pmatrix}
\]
and had preference index $11$. We begin by drawing the directed graph of this matching:
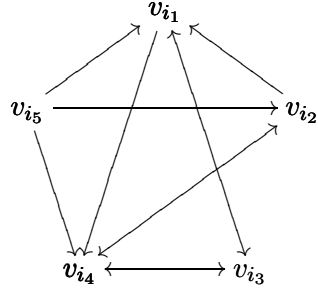
\[ \bfig
\node v_{i_1}(0,500)[v_{i_1}]
\node v_{i_5}(-475,154)[v_{i_5}]
\node v_{i_2}(475,154)[v_{i_2}]
\node v_{i_3}(294,-404)[v_{i_3}]
\node v_{i_4}(-294,-404)[v_{i_4}]
\arrow/<->/[v_{i_1}`v_{i_3};]
\arrow[v_{i_1}`v_{i_4};]
\arrow/<->/[v_{i_2}`v_{i_4};]
\arrow[v_{i_2}`v_{i_1};]
\arrow/<->/[v_{i_3}`v_{i_4};]
\arrow[v_{i_5}`v_{i_1};]
\arrow[v_{i_5}`v_{i_2};]
\arrow[v_{i_5}`v_{i_4};]
\efig \]
This looks very similar to the directed graph of $SCP_2$; the only difference is that the edge between $v_{i_1}$ and $v_{i_3}$ is double sided.
Once again no path goes through $v_{i_5}$. Therefore the trading cliques afforded by $SCP_3$ will be all the trading cliques afforded by $SCP_2$, with the addition of the following:
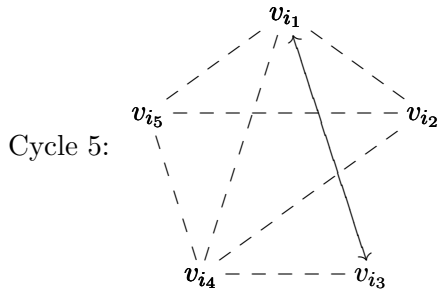
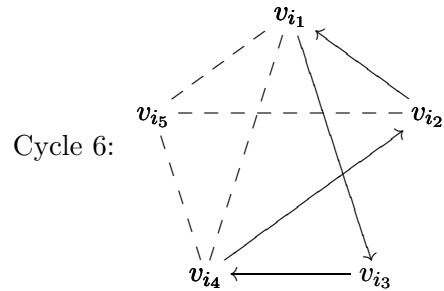
\[ \text{Cycle 5: } \bfig
\node v_{i_1}(-250,500)[v_{i_1}]
\node v_{i_5}(-725,154)[v_{i_5}]
\node v_{i_2}(225,154)[v_{i_2}]
\node v_{i_3}(44,-404)[v_{i_3}]
\node v_{i_4}(-544,-404)[v_{i_4}]
\arrow/<->/[v_{i_1}`v_{i_3};]
\arrow/--/[v_{i_1}`v_{i_4};]
\arrow/--/[v_{i_2}`v_{i_4};]
\arrow/--/[v_{i_2}`v_{i_1};]
\arrow/--/[v_{i_4}`v_{i_3};]
\arrow/--/[v_{i_5}`v_{i_1};]
\arrow/--/[v_{i_5}`v_{i_2};]
\arrow/--/[v_{i_5}`v_{i_4};]
\efig \qquad \qquad
 \text{Cycle 6: } \bfig
\node v_{i_1}(0,500)[v_{i_1}]
\node v_{i_5}(-475,154)[v_{i_5}]
\node v_{i_2}(475,154)[v_{i_2}]
\node v_{i_3}(294,-404)[v_{i_3}]
\node v_{i_4}(-294,-404)[v_{i_4}]
\arrow[v_{i_1}`v_{i_3};]
\arrow/--/[v_{i_1}`v_{i_4};]
\arrow[v_{i_2}`v_{i_1};]
\arrow[v_{i_3}`v_{i_4};]
\arrow[v_{i_4}`v_{i_2};]
\arrow/--/[v_{i_5}`v_{i_1};]
\arrow/--/[v_{i_5}`v_{i_2};]
\arrow/--/[v_{i_5}`v_{i_4};]\efig
\]
The improved matchings associated to these are:
 \[ M_1 =
\begin{array}{ccc}
i_1: s_1 \succ \underline{s_2} \succ s_5 \succ s_4 \succ s_3  \\
i_2:s_2 \succ s_5 \succ \underline{s_4} \succ s_1 \succ s_3 \\
i_3:\underline{s_5} \succ s_2 \succ s_1 \succ s_3 \succ s_4 \\
i_4:s_4 \succ \underline{s_1}\succ s_2 \succ s_3 \succ s_5 \\
i_5:s_5 \succ s_4 \succ s_2 \succ \underline{s_3} \succ s_1
\end{array} \qquad
M_2 = \begin{array}{ccc}
i_1:  s_1 \succ s_2 \succ \underline{s_5}\succ s_4 \succ s_3\\
i_2: s_2 \succ s_5 \succ \underline{s_4} \succ s_1 \succ s_3 \\
i_3: :s_5 \succ \underline{s_2} \succ s_1 \succ s_3 \succ s_4 \\
i_4: s_4 \succ \underline{s_1} \succ s_2 \succ s_3 \succ s_5 \\
i_5: s_5 \succ s_4 \succ s_2 \succ \underline{s_3}\succ s_1
\end{array}\]
\[M_3 = \begin{array}{ccc}
i_1: s_1 \succ s_2 \succ \underline{s_5}\succ s_4 \succ s_3 \\
i_2:\underline{s_2} \succ s_5 \succ s_4 \succ s_1 \succ s_3 \\
i_3:s_5 \succ s_2 \succ \underline{s_1} \succ s_3 \succ s_4 \\
i_4:\underline{s_4}\succ s_1 \succ s_2 \succ s_3 \succ s_5 \\
i_5:s_5 \succ s_4 \succ s_2 \succ \underline{s_3} \succ s_1
\end{array} \qquad
M_4 = \begin{array}{ccc}
i_1: s_1 \succ \underline{s_2} \succ s_5 \succ s_4 \succ s_3 \\
i_2: s_2 \succ \underline{s_5} \succ s_4 \succ s_1 \succ s_3 \\
i_3: s_5 \succ s_2 \succ \underline{s_1} \succ s_3 \succ s_4 \\
i_4: \underline{s_4} \succ s_1 \succ s_2 \succ s_3 \succ s_5 \\
i_5: s_5 \succ s_4 \succ s_2 \succ \underline{s_3}\succ s_1
\end{array}\]
\[M_5 = \begin{array}{ccc}
i_1: \underline{s_1} \succ s_2 \succ {s_5}\succ s_4 \succ s_3 \\
i_2: {s_2} \succ s_5 \succ \underline{s_4} \succ s_1 \succ s_3 \\
i_3: \underline{s_5} \succ s_2 \succ {s_1} \succ s_3 \succ s_4 \\
i_4: {s_4}\succ s_1 \succ \underline{s_2} \succ s_3 \succ s_5 \\
i_5:s_5 \succ s_4 \succ s_2 \succ \underline{s_3} \succ s_1
\end{array} \qquad
M_6 = \begin{array}{ccc}
i_1: \underline{s_1} \succ s_2 \succ {s_5}\succ s_4 \succ s_3 \\
i_2: {s_2} \succ \underline{s_5} \succ {s_4} \succ s_1 \succ s_3 \\
i_3: {s_5} \succ \underline{s_2} \succ {s_1} \succ s_3 \succ s_4 \\
i_4: \underline{s_4}\succ s_1 \succ {s_2} \succ s_3 \succ s_5 \\
i_5:s_5 \succ s_4 \succ s_2 \succ \underline{s_3} \succ s_1
\end{array}\]
with preference indices $7$, $9$, $7$, $7$, $7$, and $5$, respectively.
Note that none of these matchings is efficient. In other words if we run the TADAM algorithm for this problem, no matter which trading clique is picked at Round 1, we will be able to find a second trading clique to continue the process. For instance the directed graph for $M_6$ contains a single two-sided edge between $v_{i_2}$ and $v_{i_3}$, and three unrequited edges coming out of $v_{i_5}$:
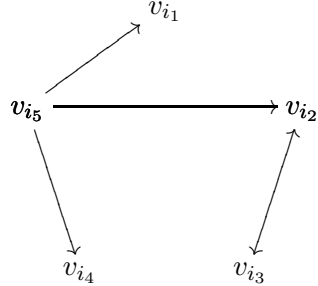
\[ \bfig
\node v_{i_1}(0,500)[v_{i_1}]
\node v_{i_5}(-475,154)[v_{i_5}]
\node v_{i_2}(475,154)[v_{i_2}]
\node v_{i_3}(294,-404)[v_{i_3}]
\node v_{i_4}(-294,-404)[v_{i_4}]
\arrow/<->/[v_{i_2}`v_{i_3};]
\arrow[v_{i_5}`v_{i_1};]
\arrow[v_{i_5}`v_{i_2};]
\arrow[v_{i_5}`v_{i_4};]
\efig \]
which yields a single trading clique between $i_2$ and $i_3$. Applying the TADAM algorithm one round further, we get the effective matching:
\[ M_E^{SCP_3} =
\begin{pmatrix}i_1&i_2&i_3&i_4&i_5 \\ s_1&s_2&s_5&s_4&s_3 \end{pmatrix}
\]
with preference index $3$. This is precisely the same matching as found by EADAM with full consent or alternatively with the coalition described in \S\S\ref{SS:CoalitionsEADAM}.

%%%%%%%%%%%%%%%%%%%
%%%%%%%%%%%%%%%%%%%
%%%%%%%%%%%%%%%%%%%

\subsection{Properties of TADAM}
\label{SS:CliquesProperties}

We begin this section with an analysis of the performance of TADAM under strategic action. We first state a key result from \cite{Ke10}:
\begin{prop}[Prop.4 in \cite{Ke10}] No Pareto efficient mechanism that can Pareto improve upon SOSM is fully immune to strategic action.\end{prop}

Since TADAM produces Pareto improvements of SOSM, it follows then that it is not strategy-proof. This is consistent with other improvements upon SOSM. However, lack of strategy-proofness does not imply easy manipulability. The feasibility of manipulation decreases as the size of the market (school district) increases. This is analogous to our earlier assertion that substantial coalitions are hard to form naturally on their own in the context of SCP. Students do not have complete information about preference profiles of other students, so potential profitable strategic behaviors are highly unlikely. Formulating an alternative ranked list which yields a better assignment, even with complete information on all other students will most likely not be feasible for individual students. 

Making the above more precise in technical language we first split the schools into categories in terms of perceived quality. Then we can prove (cf.\ Theorem 2 of \cite{Ke10}, see Appendix):
\begin{theorem}
\label{T:StrategyKesten}
Let the set of schools $S$ be partitioned into categories of perceived quality: 
\[ S =  S_1 \cup  S_2 \cup \cdots \cup S_m, \qquad \textmd{ such that } \qquad S_i \cap S_j = \emptyset \textmd{ if } i \neq j,\] 
such that for any $k,l \in \{1,\cdots,m\}$ with $k < l$, each student prefers any school in $S_k$ to any school in $S_l$. Let each student's information be \textit{symmetric} for any two schools in the same perceived quality category. Then for any student the strategy of truth telling stochastically dominates any other strategy when other students behave truthfully. Thus truth telling is an ordinal Bayesian Nash equilibrium of the preference revelation game under TADAM.
\end{theorem}

A well-studied method of strategic action by students is truncation manipulation, one of the few tools available in such a largely incomplete information matching game \cite{Eh08}. However it is easy to see that in TADAM no student benefits from truncating her preference list; any such truncation results in fewer cliques and fewer opportunities for that student (and for others) to improve her lot.

Note also that TADAM is \textit{group} strategyproof because every possible Pareto improvement has been made and no possible coalition can be formed.

Another prominent feature of TADAM is its efficiency. Each trading clique followed improves the efficiency of the outcome, neutralizing to an extent the inefficiency caused by SOSM. As each such improvement creates a Pareto domination of the previous matching and decreases the preference reverence index, at the end of the algorithm, we stop at a Pareto efficient matching. In fact TADAM produces all efficient matchings that Pareto dominate SOSM. We prove a slightly stronger result:
\begin{prop}If matching $M$ (Pareto) dominates the SOSM matching, $M^{*}$, then $M$ is realizable by TADAM up to null cliques. \end{prop}

\begin{proof} Let $E\subseteq I$ be the subset of students assigned to different schools under $M$ and $M^{*}$. For each element $i \in E$, denote by $\pi_E(i)$ the student $i^*$ with $M^*[i^*] = M[i]$; $\pi_E(i)$ is the student whose SOSM assignment is assigned to $i$ via the new matching. Note that $\pi_E$ is a permutation on $E$. 
[Recall again the fact that permutations are products of disjoint cycles (cf.\ proof of Theorem \ref{T:CoalitionEADAM}).]

Theorem \ref{T:paretocycle} implies that $M^{*}$ admits a trading clique; we want to make sure we will be aiming for $M$. To this goal, we use the following steps to identify which cliques should be used in TADAM: 
\begin{itemize}
\item \underline{Step 1}: Pick a student $i \in E$ and label her $i_{1,1}$. Then let $i_{1,2}$ be the student $\pi_E(i_{1,1})$ and more generally label $i_{1,j+1} = \pi_E(i_{1,j})$. This process will stop at some $j_1$ with $\pi_E(i_{1,j_1}) = i_{1,1}$ as $\pi_E$ is a finite permutation. Then $E_1 = ( i_{1,1}, i_{1,2},  \cdots, i_{1,j_1})$ is a clique. 
\item[] And in general,
\item \underline{Step k}, k $\ge$ 1: Pick a student $i \in E$ who has not yet been assigned to a cabal loop and label her $i_{k,1}$. If none exists then the algorithm stops. Otherwise, label $\pi_E(i_{k,1})$ as $i_{k,2}$ and more generally label $i_{k,j+1} = \pi_E(i_{k,j})$. This process stops at some $j_k$ with $\pi_E(i_{k,j_k}) = i_{k,1}$ as $\pi_E$ is finite. Then $E_k = (i_{k,1}, i_{k,2},  \cdots, i_{k,j_k}, i_{k,1})$ is a clique.
\end{itemize}
The process has to stop since $E$ is finite. Furthermore each student in $E$ shows up in exactly one step and hence in exactly one clique, because $\pi_E$ is invertible. Note that at least one of the cliques is not null (in other words at least one of the cliques is a trading clique) because for $M$ to Pareto dominate $M^*$, at least one student must strictly prefer her $M$-assignment to her $M^*$-assignment; this gives us an edge with weight $1$. Then to obtain $M$ from $M^*$ via TADAM we run TADAM using the trading cliques we have found above, and finally modify with the null cliques remaining. 
\end{proof}

Obviously, distinct Pareto efficient matchings are Pareto incomparable. At this point we might resort to another evaluative criterion. For instance we may wish to then consider the matchings with minimal preference index; this can reduce our option size.

The above proposition easily yields the following:
\begin{cor}  All efficient outcomes of EADAM and CIM can be found by TADAM.
\begin{proof}
This follows immediately from the previous proposition and noting that EADAM and CIM (Pareto) dominate SOSM. 
\end{proof}
\end{cor}

Recall that both EADAM and coalitional improvements provide us with one way to improve SOSM. However, TADAM can return all Pareto efficient matchings that dominate SOSM so that we can compare all choices and pick the most desirable matching. 

Note that since TADAM outcomes are always Pareto improvements on the SOSM baseline, TTC outcomes may not always be reachable via TADAM. This is because in some cases (e.g., $SCP_4$, $SCP_5$) the TTC outcome is Pareto incomparable to the SOSM outcome. Compare with \S\S\ref{SS:CoalitionsTTC}.

The absolute efficiency of TADAM may appeal to a utilitarian. However, this efficiency is achieved at the expense of stability. Clearly TADAM is not stable; just like TTC,
TADAM trades stability for efficiency. 
Obviously we need to make an effort to coordinate the tradeoff between stability and efficiency. 
In the school choice literature, ``fairness", ``stability", ``justified envy'', and ``no priority violation" are often used interchangeably. Here we propose a more nuanced notion of fairness.

Since TADAM starts with the SOSM outcome as input, we are starting at a point where student priorities are considered and respected. TADAM may then make  changes to the assignments which cause instability, manifesting itself in terms of justified envy. However, if a student's assigned school could not get any better under any stable mechanism, we surmise that his ``justified envy" for anybody's assignment should not be justified. To formalize this we make the
\begin{definition}
A matching is {\it reasonably fair} if there is no stable matching that can improve the assignment of any student. A mechanism is {\it reasonably fair} if it always outputs reasonably fair matchings. 
\end{definition}
Then the following is a direct consequence:
\begin{prop}
TADAM is a reasonably fair mechanism.
\end{prop}

Finally note that indifferences in student preferences are seamlessly incorporated into the cliques model (via the weights on our graph). This makes our model a versatile theoretical tool for further investigations in the SCP. Consider for instance the simple SCP below:
\[ SCP_6 : \qquad \begin{array}{ccc}
i_{1}:s_{1} = s_{2} \succ s_{3} & \qquad & s_{1}: i_{1} \succ i_{3} \succ i_{2} \\
i_{2}: s_{1}  \succ  s_{2} \succ s_{3} & \qquad & s_{2}:i_{2} \succ i_{1} \succ i_{3} \\
i_{3}: s_{1} \succ s_{2} \succ s_{3} & \qquad &  s_{3}:i_{2} \succ i_{1} \succ i_{3}
\end{array}\]
where we denote indifference by $=$.
When we run the SOSM we need to break the tie in student $i_1$'s preferences, so in one such tie-breaking step, the resulting problem equals $SCP_1$:
\[ SCP_1 : \qquad \begin{array}{ccc}
i_{1}:s_{2} \succ s_{1} \succ s_{3} & \qquad & s_{1}: i_{1} \succ i_{3} \succ i_{2} \\
i_{2}: s_{1}  \succ  s_{2} \succ s_{3} & \qquad & s_{2}:i_{2} \succ i_{1} \succ i_{3} \\
i_{3}: s_{1} \succ s_{2} \succ s_{3} & \qquad &  s_{3}:i_{2} \succ i_{1} \succ i_{3}
\end{array}\]
Then the SOSM outcome is
\[ M_S^{SCP_6}  = \begin{pmatrix}i_{1}& i_{2}& i_{3}\\ s_{1}&s_{2}&s_{3}\end{pmatrix},\]
TADAM allows us to create a trading clique between $i_1$ and $i_2$, resulting in a more optimal outcome:
\[ M_T^{SCP_6}  = \begin{pmatrix}i_{1}& i_{2}& i_{3}\\ s_{2}&s_{1}&s_{3}\end{pmatrix}.\]
This outcome has lower preference index ($2$ as opposed to $3$) and is efficient. 

Although a considerable amount of research has been done regarding indifferences within school priority classes, indifference in student preferences has not been closely considered. As far as we know, our work here provides the first cycle improvement model applied to the SCP that explicitly takes into account student preference indifferences. We intend to follow this path of inquiry further in future work. For indifferences in the context of the stable marriage problem, see \cite{Ir94}, \cite{Man02}.

%%%%%%%%%%%%%%%%%%%
%%%%%%%%%%%%%%%%%%%
%%%%%%%%%%%%%%%%%%%

\section{Conclusion}
\label{S:Conclusion}

In this paper, we introduce and investigate the properties of coalitions and cliques, two notions that can be incorporated into a school choice mechanism to improve the efficiency of SOSM. While ``coalitions" and ``cliques" might seem semantically indistinguishable, we believe the two perspectives provide different insights about cyclic improvements.

Coalitions present an opportunity for ex-ante efficiency gains upon SOSM. By effectively ignoring irrelevant preferences under SOSM, we allow for maximum gains for students under SOSM. Since matching students under SOSM in the presence of interruptors hinders efficiency, we effectively prevent students' preferences from unnecessarily conflicting with each other. 

Conversely, the notion of trading cliques presents an opportunity for ex-post efficiency gains upon SOSM. 
More specifically, TADAM realizes all possible efficiency gains that allow mutually beneficial trade. Since matching students under SOSM guarantees students the best matching they could have received under a stable mechanism,  we attempt to improve upon this baseline to provide students with an optimal matching. Fortunately, TADAM allows the policy maker to decide \emph{which} improvement is most desirable, given their socially relevant circumstances.

Our work may also be viewed as a fresh examination of three well-known and widely used school choice mechanisms (SOSM, TTC, and EADAM) each of which has well-understood strengths and weaknesses.  Our view throughout is that strict adherence to stability at the cost of student preference optimization is less than ideal and can be overcome by considering new criteria for evaluation and/or different justifications for sacrificing stability. Our introduction of the notion of ``reasonably fair'' captures this alternate focus on improving outcomes for students without bowing to the pressures of perceived unfairness.  The double meaning of reasonableness as ``somewhat'' as well as ``what a reasonable person would accept'' is especially apropos. The constructions here yield opportunities to improve upon these mechanisms while justifying resulting priority violations in new ways.

The first improvement presented here involves the creation of cheating coalitions to improve on SOSM, analogous to the use of such cheating coalitions in the stable marriage problem.  We show that using coalitions, we can obtain the same outcome as EADAM, with full consent.  Because the creation of cheating coalitions requires complete knowledge of all preference profiles, it suggests the idea that school districts explore all possible outcomes of coalition cheating in order to establish the best overall matching in terms of student preferences as well as school priorities.  This sets up the notion of school districts as ``co-conspirators'' with the additional position of benevolence in the sense of wanting the best outcome for all students in contrast to the assumed selfish goal of individual students.  The priority violations that arise as a result of EADAM (and other coalition cheating outcomes) are better justified in this context as the overall process seeking the best outcome for most.  Students would not have to be asked to waive priorities because the investigation of outcomes resulting from cheating coalitions would be considered part of the overall mechanism.  Since the cheating coalitions (as with EADAM) do not result in any students receiving a worse matching than they would have in any case, the objections should be minimized.

Our second approach to improving upon SOSM involves the introduction of trading cliques.  This process produces all possible Pareto efficient matchings that Pareto dominate SOSM.  The improvement cycles used do not result in any student being matched to a lower preference school and do not require students to waive their (perceived) priorities.  Again, in this context, the examination of all outcomes that result from applying trading cliques to the outcomes of SOSM would be considered as part of the mechanism, applied by the district (Mechanism Designer), and thus the priority violations would be more fully justified.  While there already exists a mechanism (EADAM\ \cite{Ke10}) which finds \emph{one} Pareto efficient, Pareto domination of the popular deferred acceptance mechanism (SOSM \cite{AbSo03}), TADAM  finds \emph{all} possible Pareto dominations of SOSM. This mechanism affords policy makers the luxury of comparing all available Pareto improvements of SOSM. In addition, recall that TADAM produces, among its matchings, the EADAM matching (with full consent).Ê Because TADAM is easier to run and explain than EADAM, we contend that it is a preferable way to obtain that matching.Ê Moreover, families do not have to waive priorities in the context of TADAM.Ê

Clearly our two modifications work by Pareto improving the baseline outcome of SOSM.  Since no student is ever matched to a lower ranked school than SOSM, the objections to the followup process should be minimized.  Starting with SOSM as a basis is useful because it is a mechanism that is currently in use. Thus these improvements can have genuine practical implications. We can justify the priority violations that result from coalition cheating and trading cliques by showing that the new assignments (Pareto) dominate the SOSM assignments.  Because many of the current school priorities in place are meant to create some certainty/security for families, once those have been taken into account in the initial assignment, and since we can demonstrate that no families are made worse off, neither schools nor families should have a reason to object. 

As a final note we point out that in fact the two notions introduced in this paper are related. More specifically given a coalition $C = (K,A(K))$ in the notation of \S\S\ref{SS:CoalitionImplementation}, we can always construct a sequence of trading cliques that under TADAM yields the same outcome. In other words coalitional outcomes can always be obtained via TADAM as well. Going the other way is also doable in the case of strict preference profiles: any clique in such a context corresponds to a cabal cycle and the accomplices may be determined afterwards by looking at the resulting priority violations. Nonetheless, even though the two threads of this work lead us to almost equivalent end points, we believe that their separate (but connected) treatment may encourage a more nuanced discussion.

%%%%%%%%%%%%%%%%%%%%%%%%%%%%%%%%%%%%%
%%%%%%%%%%%%%%%%%%%%%%%%%%%%%%%%%%%%%
%%%%%%%%%%%%%%%%%%%%%%%%%%%%%%%%%%%%%
%%%%%%%%%%%%%%%%%%%%%%%%%%%%%%%%%%%%%
%%%%%%%%%%%%%%%%%%%%%%%%%%%%%%%%%%%%%

\bibliography{SchoolChoiceBibliography}

\appendix

\section{Proof of Theorem \ref{T:StrategyKesten}}

\subsection{Background and notation} In the following, we describe, in our notation, the precise formal setting of Kesten's Apprendices E-F \cite{Ke10}, which in turn follow \cite{Eh08} and \cite{RR99}.

Recall from \S\S\ref{SS:Notation} that we denote a set consisting of preference profiles for each student in $I$ by ${\bf P} = \{P_i : i \in I\}$ and 
a set consisting of priority structures for each school in $S$ by ${\bf \Pi} = \{\Pi_s : s \in S\}$.
We also denoted the space of all sets ${\bf P}$, and the space of all sets ${\bf \Pi}$, by $\mathfrak{P}$ and $\mathfrak{\Pi}$, respectively.

In the following we will also need to introduce the notation ${\bf P}_i$, which stands for the class of all possible preference listings for student $i$. Set 
\[ {\bf X}_{-i} \defequal \mathfrak{\Pi}  \times \{ {\bf P_{i^{\prime}}}\}_{i^{\prime} \neq i}. \]
Following Kesten we define a \textbf{random school choice problem} (RSCP) to be a probability distribution $\hat{\bf P}_{-i}$ over ${\bf X}_{-i}$, which is intended to denote student $i$'s belief about the preference lists submitted by the rest of the students together with the priority structures given for all the schools. Analogously we define a \textbf{random matching} $\hat{{M}}$ to be a probability distribution over the set $\mathfrak{M}$ of all matchings. Given a matching mechanism $\mathcal{M}$ and a SCP $({P}_i, {P}_{-i})$, where $P_{-i} \in {\bf X}_{-i}$,
$\mathcal{M}({P}_i, {P}_{-i})$ will stand for the matching selected by $\mathcal{M}$ for this problem.

Now given a mechanism $\mathcal{M}$, and a student $i$ with preference $P_i$, each RSCP $\hat{\bf P}_{-i}$ induces a random matching $\mathcal{M}(P_i,\hat{\bf P}_{-i})$ in the following manner: For $M \in \mathfrak{M}$, set the probability $Pr(\mathcal{M}(P_i,\hat{\bf P}_{-i}) = M)$ that the mechanism outputs matching $M$ for the setup at hand equal to the probability that $\hat{\bf P}_{-i} = P_{-i}$ and $\mathcal{M}(P_i,P_{-i}) = M$. Finally define $\mathcal{M}(P_i,\hat{\bf P}_{-i})[i]$ to be the distribution that this random matching induces over student $i$'s set of placements.

We will say that, given a student $i$, her true preference list $P_i$, and two strategies (declared preference lists) $P^{\prime}_i, P^{\prime \prime}_i$ for $i$, strategy $P^{\prime}_i$ \textbf{stochastically dominates} $P^{\prime\prime}_i$ if the probability distribution induced on the placements of student $i$ when she declares $P^{\prime}_i$ stochastically dominates the probability distribution induced when she declares $P^{\prime\prime}_i$, where we base the comparison on her true preference list $P_i$. More precisely, given student $i \in I$, preference lists $P_i, P^{\prime}_i, P^{\prime \prime}_i \in {\bf P}_i$, and a RSCP $\hat{\bf P}_{-i}$, strategy $P^{\prime}_i$ \textbf{stochastically $P_i$-dominates} strategy $P^{\prime\prime}_i$ if for all $s \in S$, the probability that $i$ will be assigned by the mechanism to a school that (according to her true preference list $P_i$) she prefers over $s$ when she submits the preference list $P^{\prime}_i$ is greater than or equal to the probability of the same happening when she submits the list $P^{\prime\prime}_i$.

So let us now be given a SCP of the form $({\bf P}, {\bf \Pi}) = (\{P_i : i \in I\}, \{\Pi_s, s \in S\}$. 
Let the set of schools $S$ be partitioned into categories of perceived quality (``\emph{communal perceived quality classes}''): 
\[ S =  S_1 \cup  S_2 \cup \cdots \cup S_m, \qquad \textmd{ such that } \qquad S_i \cap S_j = \emptyset \textmd{ if } i \neq j,\] 
%such that 
where all students prefer any school in $S_k$ to any school in $S_l$, for any $k,l \in \{1,\cdots,m\}$ with $k < l$.

\subsection{Symmetry, anonymity, and positive association}
If a student cannot distinguish between two schools $s, s^{\prime}$ in terms of how other students rank them, then we say that she has \textbf{symmetric} information about the two schools. In the setting above where students all agree on the perceived quality categories of schools, the additional assumption of symmetry is quite reasonable, especially in the context of larger school districts.\footnote{Following Kesten's Footnote 27, we add another condition: Two schools assumed to be symmetric also have similar capacities. Variances in capacity may cause strategizing among groups of students, and may break the symmetry.} So we assume that each student's information be symmetric for any two schools in the same perceived quality category. 
To show that for any student the strategy of truth telling stochastically dominates any other strategy when other students behave truthfully, we refer to Ehlers' two conditions: \textit{anonymity} and \textit{positive association} \cite{Eh08}. 

To define these three terms precisely, we need some more notation first. 

%Let $\mathfrak{M}$ be the set of all matchings. 
Given a matching ${M} \in \mathfrak{M}$, and $s, s^{\prime} \in S$, denote by ${M}^{s \leftrightarrow s^{\prime}}$ denote the matching that switches the assignments of the schools $s$ and $s^{\prime}$. In other words, 
\[ {M}^{s \leftrightarrow s^{\prime}}[i] = \begin{cases}
{M}[i] & \textmd{ if } {M}^{s \leftrightarrow s^{\prime}}[i] \neq s \textmd{ or } s^{\prime} \\
s^{\prime} & \textmd{ if } {M}[i] = s \\
s & \textmd{ if } {M}[i] = s^{\prime}. 
\end{cases}\]
Similarly we can define $P_i^{s \leftrightarrow s^{\prime}}$ to be the preference list for $i$ obtained from $P_i$ by switching the order of $s$ and $s^{\prime}$ while leaving all the other rankings fixed. In this same scenario, $P_{-i}^{s \leftrightarrow s^{\prime}} \in {\bf X}_{-i}$ will denote the complementary profile for $i$, where each student $i^{\prime} \neq i$ switches the positions of $s$ and $s^{\prime}$ in his preference list, and the two schools $s$ and $s^{\prime}$ switch with one another their capacity and their priorities: $\Pi_s$ becomes the new priority for $s^{\prime}$ while $\Pi_{s^{\prime}}$ is the new priority structure for $s$. 

Finally we can make the following 
\begin{definition}
For student $i \in I$ and schools $s, s^{\prime} \in S$, $i$'s information on $s$ and $s^{\prime}$ is \textbf{symmetric} if $P_{-i}$ and $P_{-i}^{s \leftrightarrow s^{\prime}}$ are equally probable; i.e., $Pr(\hat{\bf P}_{-i}= P_{-i}) = Pr(\hat{\bf P}_{-i} = P_{-i}^{s \leftrightarrow s^{\prime}})$. $i$ has \textbf{completely symmetric} information if $i$ has symmetric information on $s$ and $s^{\prime}$ for any pair $s,s^{\prime} \in S$.
\end{definition}

\begin{definition}
A mechanism $\mathcal{M}$ is said to satisfy \textbf{anonymity} if for any student $i\in I$, $P_i \in {\bf P}_i$, ${P}_{-i} \in {\bf X}_{-i}$, and $s, s^{\prime} \in S$, whenever $\mathcal{M}(P_i,P_{-i}) = M$, we have $\mathcal{M}(P_i^{s \leftrightarrow s^{\prime}},P_{-i}^{s \leftrightarrow s^{\prime}}) = M^{s \leftrightarrow s^{\prime}}$.
\end{definition}

\begin{definition}
A mechanism $\mathcal{M}$ is said to satisfy \textbf{positive association} if for any student $i\in I$, $P_i \in {\bf P}_i$, ${P}_{-i} \in {\bf X}_{-i}$, and $s, s^{\prime} \in S$, whenever $\mathcal{M}(P_i,P_{-i})[i] = s$ and $i$ prefers $s^{\prime}$ to $s$ under the preference list $P_i$, we have $\mathcal{M}(P_i^{s \leftrightarrow s^{\prime}},P_{-i})[i] = s$.
\end{definition}

In intuitive language symmetric information means that student $i$ cannot see any difference between the two schools with respect to their priorities and the remaining student preferences. Anonymity basically says that the mechanism treats schools equally. Positive association implies that if a student decides to report a higher ranking for a school he is to be placed into in one particular problem, the mechanism will not modify his assignment. 

Ehlers \cite{Eh08} proves that DA/SOSM mechanism satisfies both conditions of anonymity and positive association. Kesten \cite{Ke10} extends these ideas to show that, under the assumption of symmetric information, EADAM also satisfies both conditions.
It is clear that TADAM also satisfies the anonymity condition. In the next section we will see that it also displays the positive association property. 
Theorem \ref{T:StrategyKesten} will then follow as a simple application of Ehlers' Theorem 3.1. Below we translate the latter result into the school choice context and present this as:
\begin{theorem}[Ehlers \cite{Eh08}]
Let $\mathcal{M} \in \mathfrak{M}$ be an anonymous mechanism with positive association.
\begin{enumerate}
\item[(a)] If a student has symmetric information for schools $s, s^{\prime}$, any strategy that reverses the true ranking of $s$ and $s^{\prime}$ is stochastically dominated by a strategy that preserves the true ranking of $s$ and $s^{\prime}$; and
\item[(b)] If a student has completely symmetric information, any strategy that changes her true ranking of the schools is stochastically dominated by a strategy that preserves their true ranking.
\end{enumerate}
\end{theorem}
As direct corollaries of the proof [cf.\ Kesten \cite[Prop.A.1]{Ke10}], we obtain two results:
\begin{cor}
Let $\mathcal{M} \in \mathfrak{M}$ be an anonymous mechanism with positive association.
If student $i$ has symmetric information for all schools in $S^{\prime} \subset S$, any strategy that changes her true ranking of schools in $S^{\prime}$ and reflects her true preferences for $s \in S \backslash S^{\prime}$ is stochastically dominated by a strategy that preserves the true rankings for all schools in $S$.
\end{cor}
\begin{cor}
Let $\mathcal{M} \in \mathfrak{M}$ be an anonymous mechanism with positive association.
If student $i$ has true preference $P_i$ and information $\hat{\bf P}_{-i}$ under the conditions of Theorem \ref{T:StrategyKesten}, strategy $P_i$ stochastically dominates any other strategy ranking schools in $S_k$ above schools in $S_l$ for all $k < l$. 
\end{cor}

\subsection{TADAM and positive association} In order to show that TADAM satisfies positive association, we will first prove:
\begin{prop}
In the setup of Theorem \ref{T:StrategyKesten}, any trading clique formed at any stage of TADAM consists only of schools within the same class of perceived quality.
\end{prop}
\begin{proof}
Let $(i_1, i_2, \cdots, i_n)$ be a trading clique formed at stage $t$ of TADAM. Without loss of generality assume that $w(e_{i_1,i_2}) = 1$, $M_{t-1}[i_1] = s_1$, $M_{t-1}[i_2] = s_2$, $s_1 \in S_{k_1}$. If $s_2 \in S_{k_2}$ for some $k_2$, then we must have $k_2 \le k_1$ because $s_2 \succ_{i_1} s_1$. 
Similarly we can show that, for all $r \in \{2, \cdots , n\}$, the communal perceived quality of school $M_{t-1}[i_r]$ must be ranked higher than or the same as the communal perceived quality of school $M_{t-1}[i_{r-1}]$. More precisely, if we define $k_r \in \N$ for $r \in \{2, \cdots , n\}$ so that $M_{t-1}[i_r] \in S_{k_r}$, we must have $k_r \le k_{r-1}$ for all $r$. But this will carry over to the end of the loop, so we also have $k_1 \le k_n$.  
As $\N$ is well-ordered, 
we need to have $k_1 = k_2 = \cdots = k_n$ to avoid inconsistency. This completes the proof.
\end{proof}

A direct corollary adapts Kesten's Lemma A.3 to TADAM:
\begin{cor}In the setup of Theorem \ref{T:StrategyKesten}, TADAM places each student to a school whose communal perceived quality class is the same as her DA/SOSM assignment.  
\end{cor}

Now we finally prove that TADAM satisfies positive association. Let $i\in I$ be an arbitrary student. Let $P_i \in {\bf P}_i$, ${P}_{-i} \in {\bf X}_{-i}$, and $s, s^{\prime} \in S$ be such that TADAM assigns $i$ to $s$ when the input problem is $(P_i,P_{-i})$. Assume further that $i$ prefers $s^{\prime}$ to $s$ under the preference list $P_i$. We need to show that if $i$ reports preference $P_i^{s \leftrightarrow s^{\prime}}$ instead, TADAM will not change her assignment. 

First of all we note that if student $i$ is not involved in any trading clique in the original problem $(P_i,P_{-i})$, she is assigned to $s$ in the step 0 run of DA/SOSM, and her changed profile will not affect her DA/SOSM outcome (as DA/SOSM satisfies positive association). Since in this new profile, she assigns a higher rank to $s$, she will not be in any new trading cliques (the graph obtained at this stage will be a proper subgraph of the original). So let us assume that student $i$ is part of a trading clique at step $t$ of TADAM running for the problem $(P_i,P_{-i})$; without loss of generality we can assume that $t$ is the smallest such. Then this means that at this stage $i$ trades her DA/SOSM assignment, $s_{DA,SOSM}$, for the problem $(P_i,P_{-i})$. Then it is clear that $s \succeq_i s_{DA/SOSM}$ according to her true preference profile $P_i$. But since according to $P_i$, $s^{\prime}$ is even more preferable, we have $s^{\prime} \succeq_i s_{DA/SOSM}$ as well. So if she reports $P_i^{s \leftrightarrow s^{\prime}}$, the directed graph formed after the first run of DA/SOSM at step 0 of TADAM will have exactly the same edges as before. However since in the original problem she was assigned to $s$ by TADAM, she was never in any clique that resulted in her assignment being switched to anything better than $s$ (from the perspective of $P_i$). If she reports $P_i^{s \leftrightarrow s^{\prime}}$, and we look at the whole run of TADAM, we can therefore see that she will once again not be in any new cliques, in particular the schools that are now ranked below $s$ but above $s^{\prime}$ will still not be involved in any cliques. When all is said and done, she will once again be in the same cliques (all with outcomes resulting in $i$ being assigned to schools that are now ranked below $s^{\prime}$, until the clique which involves her getting assigned to $s$).$\qed$
\end{document}